\documentclass[12pt, a4paper]{amsart}
\usepackage{amsmath,amsfonts,amssymb,amsthm,amscd}

\newtheorem{theorem}{Theorem}[section]
\newtheorem{proposition}[theorem]{Proposition}
\newtheorem{lemma}[theorem]{Lemma}

\theoremstyle{definition}
\newtheorem{definition}[theorem]{Definition}

\theoremstyle{remark}
\newtheorem{remark}[theorem]{Remark}

\newtheorem{problem}[theorem]{Problem}

\newcommand{\fk}{\mathfrak{k}}

\newcommand{\fp}{\mathfrak{p}}

\newcommand{\fu}{\mathfrak{u}}

\begin{document}

\title[Hamiltonian stable Lagrangian tori in $\mathbb{C}H^n$]{On Hamiltonian stable Lagrangian tori\\ in complex hyperbolic spaces}

\author[T. Kajigaya]{Toru Kajigaya}
\address{Department of mathematics, School of Engineering, Tokyo Denki University,
5 Senju Asahi-cho, Adachi-ku, Tokyo 120-8551, JAPAN\endgraf
 National Institute of Advanced Industrial Science and Technology (AIST), MathAM-OIL
}



\email{kajigaya@mail.dendai.ac.jp}


\date{\today}

\begin{abstract}
In this paper, we investigate the Hamiltonian-stability of Lagrangian tori in the complex hyperbolic space $\mathbb{C}H^n$. We consider a standard Hamiltonian $T^n$-action on $\mathbb{C}H^n$, and show that every Lagrangian $T^n$-orbits in $\mathbb{C}H^n$ is H-stable when $n\leq 2$ and there exist infinitely many H-unstable $T^n$-orbits when $n\geq 3$. On the other hand, we prove a monotone $T^n$-orbit in $\mathbb{C}H^n$ is H-stable and rigid for any $n$. Moreover, we see almost all Lagrangian $T^n$-orbits in $\mathbb{C}H^n$ are not Hamiltonian volume minimizing when $n\geq 3$ as well as the case of $\mathbb{C}^n$ and $\mathbb{C}P^n$.
\end{abstract}

\maketitle

\section{Introduction}

A Lagrangian submanifold $L$ in an almost K\"ahler manifold $(M,\omega, J)$ is called {\it Hamiltonian-minimal} ({\it H-minimal} for short, or {\it Hamiltonian stationary}) if $L$ is a critical point of the volume functional under Hamiltonian deformations. Moreover, an $H$-minimal Lagrangian is called {\it Hamiltonian-stable} ({\it H-stable} for short) if the second variation of the volume functional is nonnegative for any Hamiltonian deformations. These notions were introduced by Y.-G.Oh in \cite{Oh1} and \cite{Oh2}, and studied as a natural generalization of special Lagrangian submanifolds. We refer to \cite{AO}, \cite{MO}, \cite{Oh1}, \cite{Oh2}, \cite{Ono} and references therein for explicit examples of H-stable homogeneous Lagrangians in a Hermitian symmetric space, and \cite{JLS} for existence of H-stable Lagrangians in a general compact almost K\"ahler manifold. See also \cite{KK} for a generalization of the notion of H-stability.  
  
When $M$ is the complex Euclidean space $\mathbb{C}^n$ equipped with the standard K\"ahler structure, Oh proved that any Lagrangian torus orbit of the standard Hamiltonian $T^n$-action is H-stable in $\mathbb{C}^n$ \cite{Oh2}. Moreover, Oh conjectured that they are all {\it Hamiltonian-volume minimizing}, i.e. each torus has the least volume in its Hamiltonian isotopy class. However, using a result of Chekanov \cite{Che}, Viterbo \cite{V} first pointed out the conjecture is false for a certain torus orbit, and Iriyeh-Ono \cite{IO} showed that almost all Lagrangian torus orbits are not Hamiltonian volume minimizing, namely, the set of non Hamiltonian volume minimizing $T^n$-orbits is a dense subset in $\mathbb{C}^n$.   It is a remaining problem that a torus orbit of the form $T^k{(a,\ldots, a)}\times T^{n-k}{(b,\ldots, b)}=S^1(a)\times\cdots \times S^1(a)\times S^1(b)\times \cdots \times S^1(b)$ for $a,b>0$ and $k=1,\ldots, n$ is Hamiltonian-volume minimizing or not. 

The situation is similar when $M$ is the complex projective space $\mathbb{C}P^n$. In fact, H. Ono \cite{Ono} first proved that any Lagrangian torus orbit of the standard $T^n$-action on  $\mathbb{C}P^n$ is H-stable, however, Iriyeh-Ono showed that almost all of them are not Hamiltonian volume minimizing. The remaining case includes the Clifford torus, i.e. the unique minimal $T^n$-orbit in $\mathbb{C}P^n$, and it is conjectured that the Clifford torus is Hamiltonian volume minimizing \cite{Oh2}. Also we note that the result is generalized to some torus orbits in a general  compact toric K\"ahler manifold. See \cite{IO} for the details.

It is known that the stability of {\it minimal} Lagrangian submanifold is related to  the curvature of the ambient space. In fact,  any minimal Lagrangian submanifold in a K\"ahler manifold of {\it negative} Ricci curvature is strictly stable in the classical sense, and this is in contrast to the fact that there exists no minimal and stable Lagrangian in $\mathbb{C}P^n$ (See \cite{Oh1}).  As for the Hamiltonian stability, it is pointed out in \cite{IO} and \cite{Ono} that the isoperimetric inequality for simple closed curve implies the Hamiltonian volume minimizing property of the geodesic circle in $\mathbb{R}^2$ and $S^2$, and the problem described above can be regarded as a higher dimensional analogue in $\mathbb{C}^n$ and $\mathbb{C}P^n$, respectively.  Notice that this observation is valid even for a simple closed curve on the hyperbolic plane $\mathbb{H}^2$ since  a similar inequality holds on  $\mathbb{H}^2$  (See \cite{Osserman} or Section 4 in the present paper). However, the higher dimensional analogue of the hyperbolic case is still unknown, and this motivates us  to investigate the H-stability and Hamiltonian volume minimizing property of Lagrangian submanifold in a K\"ahler manifold of negative Ricci curvature.

A natural higher dimensional setting is to consider a compact Lagrangian submanifold in the complex hyperbolic space $\mathbb{C}H^n$. A remarkable fact for $\mathbb{C}H^n$ is that the symplectic geometry of $\mathbb{C}H^n$ is completely the same as $\mathbb{C}^n$, namely, there exists a symplectic diffeomorphism $\Phi: \mathbb{C}H^n\to \mathbb{C}^n$, and hence, any Lagrangian submanifold in $\mathbb{C}^n$ is regraded as a Lagrangian submanifold in $\mathbb{C}H^n$ by the map $\Phi$. Moreover, as pointed out in \cite{HK},  there is a correspondence between compact homogeneous Lagrangian submanifolds in $\mathbb{C}H^n$ and the ones in $\mathbb{C}^n$, and we have many examples of { H-minimal} Lagrangian in $\mathbb{C}H^n$ because any compact homogeneous Lagrangian in a K\"ahler manifold is H-minimal. We note that the compact Lagrangian is never minimal in the classical sense  because any minimal submanifold in $\mathbb{C}^n$ and $\mathbb{C}H^n$ must be non-compact.
Although some compact H-stable Lagrangian  in $\mathbb{C}^n$ are known (see \cite{AO} and \cite{Oh2}),  the stability of the corresponding Lagrangian  in $\mathbb{C}H^n$ might be different from the Euclidean case since the stability depends on the metric.   
In the present paper, we restrict our attention to the torus orbits in $\mathbb{C}H^n$, and investigate the stability.

Let us describe our main results. We equip $\mathbb{C}H^n\simeq SU(1,n)/S(U(1)\times U(n))$ with the standard K\"ahler structure $(\omega, J, g)$ of constant holomorphic sectional curvature $-4$, and regard $\mathbb{C}H^n$ as an open unit ball $B^n=\{z\in \mathbb{C}^n; |z|<1\}$ in the standard way (see Section 3). We consider the maximal torus $T^n$ of a maximal compact subgroup $K=S(U(1)\times U(n))$ of $G=SU(1,n)$.  Then the $T^n$-action on $\mathbb{C}H^n$ is Hamiltonian and the principal orbits are all Lagrangian. We take a diffeomorphism between $\mathbb{C}H^n$ and $\mathbb{C}^n$ by
\begin{align*}
\Phi: \mathbb{C}H^n\simeq B^n\to \mathbb{C}^n,\quad z\mapsto \sqrt{\frac{1}{1-|z|^2}}z.
\end{align*}
Then, it turns out that $\Phi$ is a $K$-equivariant symplectic diffeomorphism. Moreover, the $T^n$-action on $\mathbb{C}H^n$ is equivariant to the $T^n$-action on $\mathbb{C}^n$ via the symplectic diffeomorphism $\Phi$ (see Section 3 and 4). In particular, there exists a one-to-one correspondence between the $T^n$-orbits in $\mathbb{C}H^n$ and the $T^n$-orbits in $\mathbb{C}^n$. We denote the principal $T^n$-orbit in $\mathbb{C}^n$  by $T(r_1,\ldots r_n):=S^1(r_1)\times \cdots \times S^1(r_n)$, where $r_i$ is the radius of the $i$-th circle.

We say an H-stable Lagrangian is {\it rigid} if the null space of the second variation under  Hamiltonian deformations is spanned by normal projections of holomorphic Killing vector fields on $\mathbb{C}H^n$.
We show the following results:

\begin{theorem}\label{maintheorem} 

\begin{enumerate}
\item[(a)] If $n\leq 2$,  every Lagrangian $T^n$-orbits in $\mathbb{C}H^n$ is H-stable and rigid.
\item[(b)] Suppose $n\geq 3$. If there exist distinct indices $i,j,k\in \{1,\ldots,n\}$ such that the inequality
\begin{align*}
\Big(1+\sum_{l=1}^n r_l^2\Big)^{1/2}r_i<r_jr_k
\end{align*}
holds, then the $T^n$-orbit $\Phi^{-1}(T(r_1,\ldots, r_n))$ is H-unstable in $\mathbb{C}H^n$. In particular, there exist infinitely many H-unstable $T^n$-orbits in $\mathbb{C}H^n$. On the other hand, the monotone $T^n$-orbit $\Phi^{-1}(T(r,\ldots, r))$ is H-stable and rigid  in $\mathbb{C}H^n$ for any $n\geq 1$ and $r>0$.
 \item[(c)] Suppose $n\geq 3$. Then, almost all Lagrangian $T^n$-orbits are not Hamiltonian volume minimizing in $\mathbb{C}H^n$.
\end{enumerate}
\end{theorem}

 See also Proposition \ref{p3}, Theorem \ref{cliff}, \ref{n2} and \ref{nvol}  for more precise statement. Although almost all Lagrangian $T^n$-orbits are not Hamiltonian volume minimizing when $n\geq 3$, the Hamiltonian volume minimizing property of the monotone $T^n$-orbit in $\mathbb{C}H^n$ is still an open problem as well as the case of $\mathbb{C}^n$ and $\mathbb{C}P^n$ (See Section 4 for further discussion).

In general, the second variational formula of the volume functional for non-minimal, H-minimal Lagrangian submanifold $L$ under Hamiltonian deformation is described by a linear elliptic differential operator of 4th order depending on both intrinsic and extrinsic properties of the immersion, and the analysis of the operator is much difficult than the case of minimal Lagrangian  (See \cite{Oh2} or Section 2).  For the case of torus orbit in a compact toric K\"ahler manifold, Ono described the operator  by using a K\"ahler potential on a complex coordinate of the toric manifold \cite{Ono}. On the other hand, our computation method in the present paper is slightly different from \cite{Ono}. We use geometry of $\mathbb{C}H^n$, in particular, the  $K$-equivariant global symplectic diffeomorphism from $\mathbb{C}H^n$ to $\mathbb{C}^n$. This map  makes it possible to rewrite the second variation for a class of Lagrangian submanifolds in $\mathbb{C}H^n$ in terms of the corresponding geometry of $\mathbb{C}^n$ (Theorem \ref{keyprop}), so that the calculation of several geometric quantities are much easier than a direct computation by using the hyperbolic metric. We remark that, in principle, our formula can be applied to not only torus orbits, but also any compact homogeneous Lagrangian submanifold in $\mathbb{C}H^n$.  Finally, we apply the results to the torus orbits in $\mathbb{C}H^n$ and give a proof of Theorem \ref{maintheorem}.

\section{Preliminaries}\label{sec2}

In this section, we give a general description of Lagrangian submanifold with $S^1$-symmetry in a K\"ahler manifold. 

Let $M$ be a complex $n$-dimensional K\"ahler manifold with the K\"ahler structure $(\omega, J)$, where  $\omega$ is the K\"ahler form and $J$ is the complex structure, and  $\phi: L\rightarrow M$ a Lagrangian immersion of a real $n$-dimensional manifold $L$ into $M$, that is, an immersion of $L$ satisfying $\phi^*\omega=0$. We denote the compatible Riemannian metric by $g$, i.e. $g(\cdot ,\cdot)=\omega(\cdot, J\cdot)$, and we often use the same symbol $g$ for the induced metric. 

Suppose a 1-dimensional connected  subgroup $Z\subset {\rm Aut}(M, \omega, J)$ acts properly on $M$ in a Hamiltonian way, and we denote the moment map of the action by $\mu: M\rightarrow \mathbb{R}\simeq \mathfrak{z}^*$, where $\mathfrak{z}$ is the Lie algebra of $Z$. We take $c\in \mathbb{R}$ and consider the level set $\mu^{-1}(c)$. In the following, we always assume $c\in \mathbb{R}$ is a regular value for $\mu$ so that $\mu^{-1}(c)$ is a real hypersurface in $M$. Since $Z$ is abelian, one easily check that $Z$ acts on $\mu^{-1}(c)$. We denote the immersion by $\iota: \mu^{-1}(c)\rightarrow M$.   

Take a non-zero element $v\in \mathfrak{z}$ and define $\tilde{v}_p:=(d/dt)|_{t=0} {\rm exp}tv\cdot p$ the fundamental vector field of the $Z$-action at $p\in \mu^{-1}(c)$. Set $\mathfrak{z}_p:={\rm span}_{\mathbb{R}}\{\tilde{v}_p\}$. Then, the tangent space of $\mu^{-1}(c)$ is decomposed into
\begin{align}\label{dec}
T_p\mu^{-1}(c)=E_p\oplus \mathfrak{z}_p,
\end{align}
where $E_p$ is the orthogonal complement of $\mathfrak{z}_p$ in $T_p\mu^{-1}(c)$. Note that $E_p$ is a $J$-invariant subspace in $T_pM$. Moreover, we see $J\tilde{v}_p$ is a normal direction of $\mu^{-1}(c)$ in $M$. In fact, we have
\begin{align*}
g(J\tilde{v}_p, X)=\omega(\tilde{v}_p, X)=d\mu^{v}_p(X)=0
\end{align*}
for any $X\in \Gamma(T\mu^{-1}(c))$ since the $Z$-action is Hamiltonian. We set
\begin{align*}
\xi_p:=\frac{\tilde{v}_p}{|\tilde{v}_p|_g}\quad{\rm and}\quad N_p:=J{\xi}_p.
\end{align*}
 The unit vector field $\xi_p$ will be called  {\it Reeb vector field} on $\mu^{-1}(c)$, and $N$ defines a unit normal vector field on $\mu^{-1}(c)$ in $M$. Also, we define a $1$-form on $\mu^{-1}(c)$ by $\eta:=\iota^*\{g(\xi, \cdot)\}=\iota^*\{g(-JN, \cdot)\}$ so that $E_p={\rm Ker}\eta_p$.
 
It is known that if the Lagrangian immersion $\phi$ is $Z$-invariant, then there exists $c\in \mathbb{R}\simeq \mathfrak{\mathfrak{z}^*}$ so that $\phi(L)\subset \mu^{-1}(c)$. Thus, for the $Z$-invariant Lagrangian immersion $\phi:L\rightarrow \mu^{-1}(c)\subset M$, we have an orthogonal decomposition
 \begin{align*}
T_pL=E_p^l\oplus \mathfrak{z}_p\subset T_p\mu^{-1}(c),
\end{align*}
where $E_p^l$ is the orthogonal complement of $\mathfrak{z}_p$ in $T_pL$.
Note that $E_p=E_p^l\oplus JE_p^l$ since $L$ is Lagrangian. According to this decomposition, we denote the tangent vector $X\in T_pL$ by
\begin{align*}
X=X_E+\eta(X)\xi.
\end{align*}
 
Suppose $Z$ acts on $\mu^{-1}(c)$ freely. Then, the quotient space $M_c:=\mu^{-1}(c)/Z$ is a smooth manifold and the standard K\"ahler reduction procedure yields a K\"ahler structure $(\omega_c, J_c)$ on $M_c$  so that $\pi^*\omega_c=\iota^*\omega$ and $\pi_*J=J_c\circ \pi_*$, where $\pi: \mu^{-1}(c)\rightarrow M_c$ is the projection. Note that $\pi$ is a Riemannian submersion and $\pi_*|_{E_p}: E_p\xrightarrow{\sim} T_{\pi(p)}M_c$ is an isomorphism. In particular, the Levi-Civita connections $\widetilde{\nabla}$ of $(\mu^{-1}(c),g)$ and $\overline{\nabla}^c$ of $(M_c,g_c)$ are related as $\pi_*(\widetilde{\nabla}_XY)=\overline{\nabla}^c_{\pi_*X}\pi_*Y$ for any $X,Y\in \Gamma(E)$.
See \cite{Futaki} for details of K\"ahler reduction.

We denote the  shape operator of the immersion $\iota: \mu^{-1}(c)\rightarrow M$ by $\widetilde{A}: \Gamma(T\mu^{-1}(c))\rightarrow \Gamma(T\mu^{-1}(c))$, i.e., $\widetilde{A}(X):=-(\overline{\nabla}_XN)^{\top}$, where $\overline{\nabla}$ is the Levi-Civita connection on $TM$, and $\top$ means the orthogonal projection onto $T\mu^{-1}(c)$. In the present paper, we are interested in a special class of hypersurfaces so called {\it $\eta$-umbilical hypersurfaces}. Namely, we suppose the shape operator of the immersion $\iota: \mu^{-1}(c)\to M$ satisfies
\begin{align}\label{etaumb}
\widetilde{A}(X)=aX+b\eta(X)\xi.
\end{align}
for some constants $a,b\in \mathbb{R}$. Note that $a$ and $a+b$ are eigenvalues of $\widetilde{A}$ and $\xi$ gives a eigenvector for the eigenvalue $a+b$.  In this case, we have the following simple fact:
Denote the holomorphic sectional curvature tensors of $M$ and $M_c$ by $T$ and $T_c$, respectively. Then, by the result of S. Kobayashi \cite{Koba},  we have
\begin{align*}
T_c(\pi_*X)=T(X)+4g(\widetilde{A}(X),X)^2=T(X)+4a^2
\end{align*}
for any $X\in E_p$ with $|X|=1$. In particular, if $M$ is a complex space form, then the quotient space $M_c$ of the $\eta$-umbilical hypersurface also has constant holomorphic sectional curvature. Thus, if furthermore $M_c$ is simply-connected, then $M_c$ is a complex space form again.  We exhibit the concrete examples of $\eta$-umbilical hypersurafaces in complex space forms and these K\"ahler quotient spaces  in Tabel 1. We refer to \cite{Berndt}, \cite{BV}  and references therein for  details. 
\begin{table}[htb]
\center{
  \begin{tabular}{|c|c|c|c|c|}\hline
$M$ & $\mu^{-1}(c)$  & $Z$ & $a,b$ & $M_c=\mu^{-1}(c)/Z$ \\ \hline
$\mathbb{C}^n$ & 
\begin{tabular}{l}
hypersphere \\ of radius $r$
\end{tabular}
 & $S^1$ & 
\begin{tabular}{l}
 $a=1/r$ \\ $b=0$
 \end{tabular} & $\mathbb{C}P^{n-1}(4/r^2)$ \\ \hline
$\mathbb{C}P^n(4)$ &
\begin{tabular}{l}
geodesic hypersphere \\ of radius $r$
\end{tabular}
& $S^1$  & 
\begin{tabular}{l}
$a=\cot(r)$\\ $b=-\tan(r)$ 
\end{tabular}
& $\mathbb{C}P^{n-1}(4/\sin^2r)$\\ \hline
 & 
 \begin{tabular}{l}
 geodesic hypersphere \\ of radius $r$
 \end{tabular} 
  &$S^1$ 
  &
   \begin{tabular}{l}
  $a=\coth(r)$\\ $b=\tanh(r)$
  \end{tabular}
  & $\mathbb{C}P^{n-1}(4/\sinh^2r)$  \\ \cline{2-5}
$\mathbb{C}H^n(-4)$ &horosphere & $\mathbb{R}$ & $a=1, b=1$  & $\mathbb{C}^{n-1}$ \\ \cline{2-5}
 & 
  \begin{tabular}{l}
 tube of radius $r$\\ around $\mathbb{C}H^{n-1}(-4)$
 \end{tabular}
 & $S^1$ &
 \begin{tabular}{l}
 $a=\tanh(r)$\\ $b=\coth(r)$ 
  \end{tabular}
 & $\mathbb{C}H^{n-1}(-4/\cosh^2r)$
 \\ \hline
 \end{tabular}
 \caption {$\eta$-umbilical hypersurfaces in complex space forms.}
}
\end{table}

Suppose  a $Z$-invariant Lagrangian immersion $\phi: L\rightarrow M$ is contained in an $\eta$-umbilical hypersurface $\mu^{-1}(c)$.
Denote the second fundamental form of the immersions $\phi: L\rightarrow M$ and ${\phi}': L\rightarrow \mu^{-1}(c)$  by $B$ and $B'$, respectively. Also, we define the mean curvature vectors of these immersions by $H:={\rm tr}B$ and $H':={\rm tr}B'$, respectively. A direct computation shows that
\begin{align}\label{bnl}
B(X,Y)&=(\overline{\nabla}_XY)^{\perp}
=B'(X,Y)+\widetilde{B}(X,Y)
\end{align}
for any $X,Y\in \Gamma(TL)$, where $\widetilde{B}$ is the second fundamental form of $\iota: \mu^{-1}(c)\to M$. Therefore, we obtain from \eqref{etaumb} and \eqref{bnl}
\begin{align}\label{mcl}
H=H'+(an+b)J\xi.
\end{align}
Note that $H'\in JE_p^l$ and $J\xi=N$. 

We often use the following $(0,3)$-tensor field on $L$:
\begin{align*}
S(X,Y,W):=g(B(X,Y),JW)\quad {\rm for}\quad {X,Y,W}\in \Gamma(TL).
\end{align*}
We remark that the sign is different from \cite{Oh2} for the definition of $S$. It is easy to see that $S$ is symmetric for all three components by the K\"ahler condition. Since we assume $L$ is Lagrangian, $S$ and $B$ have the same information. The following lemma will be used in the next section:
 
\begin{lemma}\label{sform}
Suppose the $Z$-invariant Lagrangian submanifold $L$ is contained in an $\eta$-umbilical hypersurface $\mu^{-1}(c)$. For 
any $X\in T_pL$, we have
 \begin{align}\label{sform1}
 S(X,X, JH)&=S(X_E, X_E, JH')+2a\cdot \eta(X)g(X_E,JH')\\
&\quad-(an+b)\{a|X|^2+b\eta(X)^2\}.\nonumber
 \end{align}
\end{lemma}
\begin{proof}
By using \eqref{etaumb} and the K\"ahler condition, we note that
\begin{align*}
S(X,Y,\xi)&=S(X,\xi,Y)=g(\overline{\nabla}_X\xi, JY)=-g(\overline{\nabla}_XN,Y)\\
&=g(\widetilde{A}(X),Y)=ag(X,Y)+b\eta(X)\eta(Y)
\end{align*}
for $X,Y\in T_pL\subset T_p\mu^{-1}(c)$. In particular, we have
\begin{align*}
S(X_E, Y, \xi)=ag(X_E, Y)\quad{\rm and}\quad  S(X_E, \xi,\xi)=0.
\end{align*}
Combining this with \eqref{mcl}, we see
\begin{align*}
S(X,X,JH)&=S(X,X, JH')-(an+b)S(X,X, \xi)\\
&=S(X_E, X_E, JH')+2\eta(X)\cdot ag(X_E, JH')-(an+b)\{a|X|^2+b\eta(X)^2\}.
\end{align*}
This proves \eqref{sform1}.
\end{proof}

Recall that an infinitesimal deformation $\phi_s: L\times (-\epsilon, \epsilon)\rightarrow M$ of a Lagrangian immersion $\phi_0=\phi: L\rightarrow M$ into a (almost) K\"ahler manifold $(M,\omega, J, g)$ is called {\it Hamiltonian} if the variational vector field $V:=d\phi_s/ds|_{s=0} $ is a Hamiltonian vector field, i.e., there exists $u\in C^{\infty}(L)$ so that $\alpha_{V}:=\phi^{*}i_{V}\omega=du$. A Lagrangian immersion $\phi$ is {\it Hamiltonian-minimal} ({\it H-minimal} for short) if $d/ds|_{s=0}{\rm Vol}_g(\phi_s)=0$ for any Hamiltonian deformation $\phi_s$ of $\phi=\phi_0$, where ${\rm Vol}_g(\phi)$ is the volume of $\phi$ measured by the volume measure $dv_g$ of $g$. Moreover, an H-minimal Lagrangian is {\it Hamiltonian-stable} ({\it H-stable} for short) if $d^2/ds^2|_{s=0}{\rm Vol}_g(\phi_s)\geq 0$ for any Hamiltonian deformation $\phi_s$.  

By the result of Oh \cite{Oh2},  the H-minimality is equivalent to ${\rm div}_{g}(JH)=0$.  A typical example of H-minimal Lagrangian submanifold is obtained by a compact group action. Namely, if a compact connected Lie subgoup $G\subset {\rm Aut}(M,\omega, J)$ admits a Lagrangian orbit $G\cdot p$ for some $p\in M$, then $G\cdot p$ is always H-minimal by the divergence theorem (cf. \cite{AO}). 

For an H-minimal Lagrangian submanifold in a K\"ahler manifold, Oh proved the following  second variational formula under the Hamiltonian deformation $\phi_s$:
\begin{align}\label{svf}
\frac{d^2}{ds^2}\Big|_{s=0}{\rm Vol}_g(\phi_s)=\int_L|\Delta u|^2-\rho(\nabla u, J\nabla u)+2S(\nabla u, \nabla u, JH)+JH(u)^2dv_g,
\end{align}
where $u$ is the Hamiltonian function of the variational vector field $V$ and $\rho$ is the Ricci form of $M$ (Recall that the sign of $S$ is different from \cite{Oh2}). In the following sections, we consider the second variation \eqref{svf} in a specific situation.

\section{Lagrangian submanifolds  in $\mathbb{C}H^n$}\label{sec3}
 In this section, we consider a Lagrangian submanifold contained in a special case of $\eta$-umbilical hypersurface in the complex hyperbolic space $\mathbb{C}H^n$. The main purpose of this section is to prove Theorem \ref{keyprop}.

\subsection{Geometry of $\mathbb{C}H^n$}
Let $\mathbb{C}^n$ be the  complex Euclidean space equipped with the standard K\"ahler structure $(\omega_{0}=\frac{\sqrt{-1}}{2}\sum_{i=1}^ndz^i\wedge d\overline{z}^i, J_{0}, g_{0})$. Also, we  denote the standard Hermitian inner product and its norm on $\mathbb{C}^n$ by $\langle,\rangle$ and $|\cdot|$, respectively.

Let $\mathbb{C}^{1,n}$ be the complex Euclidean space $\mathbb{C}^{1+n}$ with the Hermitian inner product $\langle,\rangle'$ of signature $(1,n)$ and $P(\mathbb{C}^{1,n})$ is the projective space. The {\it complex hyperbolic space} $\mathbb{C}H^n$ is defined by 
\begin{equation*}
\mathbb{C}H^n:=\{[l]\in P(\mathbb{C}^{1,n}); l={\rm span}_{\mathbb{C}}\{z\}\ {\rm with}\ \langle z, z\rangle' <0\}.
\end{equation*}
In the present paper, we use the ball model for $\mathbb{C}H^n$, namely, we identify $\mathbb{C}H^n$ with the open unit ball
\begin{align*}
B^{n}:=\{z\in \mathbb{C}^n; |z|<1\}\subset \mathbb{C}^n
\end{align*}
by the map 
\begin{align}\label{bide}
B^{n}\ni z\mapsto [1: z]\in \mathbb{C}H^n\subset P(\mathbb{C}^{1,n}).
\end{align}
The standard complex structure $J_{0}$ on $\mathbb{C}^n$ defines the complex structure $J$ on $B^{n}$. Moreover, the standard K\"ahler form on $B^{n}$ (or $\mathbb{C}H^n$) is defined by
\begin{align}\label{om}
\omega&=\frac{-1}{2}\sqrt{-1}\partial\overline{\partial}\log(1-|z|^2)\\
&=\frac{1}{2}\frac{\sqrt{-1}}{(1-|z|^2)^2}\Big{\{}\Big(\sum_{i=1}^n \overline{z}^idz^i\Big)\wedge\Big(\sum_{j=1}^n z^jd\overline{z}^j\Big)+(1-|z|^2)\sum_{i=1}^n dz^i\wedge d\overline{z}^i\Big{\}}.\nonumber
\end{align}
Then, the holomorphic sectional curvature of $(B^n,\omega, J)$ is negative constant which is equal to $-4$. See \cite{G} for details.
We denote the compatible K\"ahler metric on $B^{n}$ by $g$. 

Recall that $SU(1,n)$ acts on $B^{n}$ through the map \eqref{bide}, where $SU(1,n)$ naturally acts on $\mathbb{C}^{1,n}$ and $P(\mathbb{C}^{1,n})$. Moreover, the action is transitive, and the stabilizer group at $z=0$ is given by $K=S(U(1)\times U(n))$. In particular, $\mathbb{C}H^n$ is identified with $G/K= SU(1,n)/S(U(1)\times U(n))$. Note that the stabilizer subgroup $K$ is a maximal compact subgroup of $SU(1,n)$, and it acts on $B^{n}$ by
\begin{align*}
k\cdot z=w^{-1}A z\quad {\rm for}\quad k:=
\left[
\begin{array}{c|ccc}
w & & &  \\ \hline
 & & & \\
 & & A & \\
& & & 
\end{array}
\right]\in K\ {\rm and}\ z\in B^{n}.
\end{align*}
Moreover, $K$ acts on the tangent space $T_0B^n$ by the isotropy representation $K\rightarrow U(n)$. By \eqref{om}, we see $(T_0B^n, \omega_0)$ is naturally identified with the standard symplectic vector space $(\mathbb{C}^n,\omega_0)$. Thus, $K$ acts on $\mathbb{C}^n$ by this identification.

A  principal $K$-orbit in $B^n$ coincides with a hypersphere $S^{2n-1}(R):=\{z\in B^{n}; |z|=R\}$ in $B^{n}$ of radius $R\in (0,1)$. On the other hand, one can check that the geodesic distance $r:=d(0,z)$ between $0$ and $z\in S^{2n-1}(R)$ with respect to the hyperbolic metric \eqref{om} is given by
\begin{align*}
r=d(0,z)=\tanh^{-1}(R)
\end{align*}
(See Section 3.1.7 in \cite{G} for instance.  Note that the holomorphic sectional curvature of $\mathbb{C}H^n$ is equal to $-1$ in Section 3.1.7 in \cite{G}, although we assume it is equal to $-4$).
In particular, $S^{2n-1}(R)$ is a geodesic hypersphere in $B^{n}$ of geodesic radius $r=\tanh^{-1}(R)$. Therefore, we denote the geodesic hypersphere  of geodesic radius $r\in (0,\infty)$ in $B^{n}$ by
\begin{align*}
S^{2n-1}_r:=S^{2n-1}(\tanh r)=\{z\in B^{n}; |z|=\tanh r\}.
\end{align*}
Note that the geodesic hyperspheres in $B^n$ of different radii are {\it not homothetic} to each other with respect to the induced metrics from $g$, and they are so called the {\it Berger spheres}.

Let us consider the symplectic structure of $\mathbb{C}H^n$. It is known that any Hermitian symmetric space of non-compact type is symplectic diffeomorphic to the symplectic vector space (cf. \cite{Mc}). For the case of $\mathbb{C}H^n$,  we have the following explicit identification (cf. \cite{G}, \cite{HK}):
\begin{lemma}
A map defined by 
\begin{align}\label{sympdiff}
\Phi: B^{n} \rightarrow \mathbb{C}^n,\quad z\mapsto \sqrt{\frac{1}{1-|z|^2}}\cdot z, 
\end{align}
gives a $K$-equivariant symplectic diffeomorphism, i.e. $\Phi^*\omega_{0}=\omega$. 
\end{lemma}
\begin{proof}
By the definition of $K$-actions and $\Phi$, it is easy to verify that $\Phi$ is $K$-equivariant diffeomorphism. On the other hand, a section of the cohomogeneity one $K$-action on $B^n$  is given by $\{(0,\ldots,0, R)\in B^n; R\in [0,1)\}$. Thus, in order to prove the second assertion, it is sufficient to check at a point ${\bf r}:=(0,\ldots, 0,R)\in B^n$ for $R\in [0,1)$. Note that, at the point ${\bf r}$, we have
\begin{align*}
\Phi_*\frac{\partial}{\partial x^i}\Big{|}_{{\bf r}}&=\sqrt{\frac{1}{1-R^2}}\frac{\partial}{\partial x^i}\Big{|}_{\Phi({\bf r})}\quad {\rm for}\  i\neq n, \quad 
\Phi_*\frac{\partial}{\partial x^n}\Big{|}_{{\bf r}}=\Big(\frac{1}{1-R^2}\Big)^{3/2}\frac{\partial}{\partial x^n}\Big{|}_{\Phi({\bf r})},\\
\Phi_*\frac{\partial}{\partial y^i}\Big{|}_{{\bf r}}&=\sqrt{\frac{1}{1-R^2}}\frac{\partial}{\partial y^i}\Big{|}_{\Phi({\bf r})}\quad {\rm for}\ i=1,\ldots,n,
\end{align*}
where we set $z^i=x^i+\sqrt{-1}y^i$.
On the other hand, we have $(\omega_0)_{\Phi(\bf r)}=\sum_{i=1}^ndx^i\wedge dy^i$ and 
\begin{align*}
\omega_{\bf r}=\frac{1}{1-R^2}\sum_{j=1}^{n-1} dx^j\wedge dy^j+\frac{1}{(1-R^2)^2}dx^n\wedge dy^n.
\end{align*}
By using these equalities, one can easily check that $\Phi^*\omega_0=\omega$.
\end{proof}

 In the following, we identify $B^{n}$ (or $\mathbb{C}H^n$) with $\mathbb{C}^n$ as a symplectic manifold by $\Phi$.

Let us consider the $C(K)$-action on $B^n$, where 
\begin{align*}
C(K):=\{{\rm diag}(e^{-n\sqrt{-1}\theta},e^{\sqrt{-1}\theta},\ldots, e^{\sqrt{-1}\theta}); \theta \in [0,2\pi]\}
\end{align*}
is the center of $K$. Note that $C(K)$ does not act effectively on $B^n$.  Indeed, $C(K)$ acts on $B^n$ by $z\mapsto e^{\sqrt{-1}(n+1)\theta}z$ for $z\in B^n$. In order to adapt the argument of the previous section,  we take a normal subgroup $N$ of $C(K)$:
\begin{align*}
N:=\{{\rm diag}(e^{-n\sqrt{-1}\frac{2\pi k}{n+1}},e^{\sqrt{-1}\frac{2\pi k}{n+1}},\ldots, e^{\sqrt{-1}\frac{2\pi k}{n+1}}); k=0,1,\cdots, n\}.
\end{align*}
Obviously, $N$ is isomorphic to $\mathbb{Z}_{n+1}$ and $Z:=C(K)/N$ is homeomorphic to $S^1$. Moreover, $Z$ acts  on $B^n$ effectively and freely through the $C(K)$-action. Indeed, $Z$ acts on $B^n$ by $z\mapsto e^{\sqrt{-1}\theta}z$ for $e^{\sqrt{-1}\theta}\in S^1$ via the identification $Z\simeq S^1$.

 Since $\Phi$ is a $K$-equivariant symplectic diffeomorphism, a moment map $\mu: B^{n} \rightarrow \mathbb{R}$ of the $Z$-action on ${B}^n$ is given by
 \begin{align*}
 \mu(z)=\mu_{0}\circ \Phi(z)=-\frac{1}{2}\Big(\frac{|z|^2}{1-|z|^2}\Big)
 \end{align*}
where $\mu_{0}: \mathbb{C}^n\rightarrow \mathbb{R}$ is a moment map of the $Z$-action on $\mathbb{C}^n$ which is given by $\mu_{0}(z):=-\frac{1}{2}|z|^2$. Thus, a regular level set $\mu^{-1}(c)$ coincides with a geodesic hypersphere $S^{2n-1}_r=S^{2n-1}(\tanh r)$ in $B^{n}$, that is, a $K$-orbit. 

 We fix a fundamental vector field of the $Z$-action on $B^{n}$ (or $\mathbb{C}^n$) defined by
\begin{align*}
\tilde{v}_z:=\sum_{i=1}^n \Big(-y^i\frac{\partial}{\partial x^i}+x^i\frac{\partial}{\partial y^i}\Big)=\sqrt{-1}\sum_{i=1}^n\Big(z^i\frac{\partial}{\partial z^i}-\overline{z}^i\frac{\partial}{\partial \overline{z}^i}\Big)
\end{align*}
for $z\in B^{n}$ (or $\mathbb{C}^n$) so that $J\tilde{v}_z=-{\bf p}$ is the inner position vector. On the other hand, a direct computation shows that  we have
\begin{align}\label{normv}
|\tilde{v}_z|_g=\frac{|z|}{1-|z|^2}=\frac{\tanh r}{1-\tanh^2r}={\sinh r\cosh r} 
\end{align}
for $z\in S^{2n-1}_r\subset B^n$. Note that the norm $|\tilde{v}_z|_g$ depends only on $r$, and this implies that  $Z$-orbits contained in $S^{2n-1}_r$ are mutually isometric. The Reeb vector field on $S^{2n-1}_{r}$ is given by
\begin{align*}
\xi_z:=\frac{\tilde{v}_z}{|\tilde{v}_z|_g}=\frac{\tilde{v}_z}{\sinh r\cosh r}\quad {\rm for}\quad z\in S^{2n-1}_r.\end{align*}
Note that $N:=J_{0}\xi$ is the {\it inner} unit normal vector field of $S^{2n-1}_r$.  Moreover, it is known that the shape operator $\widetilde{A}$ of  $S^{2n-1}_r\subset B^{n}$ with respect to $N$ is given by
\begin{align*}
\widetilde{A}(X)=\coth r\cdot X+\tanh r\cdot \eta(X)\xi, 
\end{align*}
namely, $S^{2n-1}_r$ is an $\eta$-umbilical hypersurface in $B^{n}$. 
In particular, the K\"ahler quotient space $\mu^{-1}(c)/Z$ is exactly the complex projective space $\mathbb{C}P^{n-1}(4/\sinh^2r)$ (see  Section 2).

\subsection{Comparison of $\mathbb{C}H^n$ with $\mathbb{C}^n$}
Let $\phi_1: L\rightarrow B^{n}$ be a $C(K)$-invariant Lagrangian embedding into $B^n$. Note that $L$ is $C(K)$-invariant if and only if so is $Z$-invariant, where $Z=C(K)/N$.  In this subsection, we shall compare geometric properties of $\phi_1$ with corresponding properties of the composition 
\begin{align*}
\phi_2:=\Phi\circ \phi_1: L\to \mathbb{C}^n.
\end{align*}
Note that $\phi_2$ is a $C(K)$-invariant Lagrangian embedding into $\mathbb{C}^n$ since $\Phi$ is a $C(K)$-equivariant symplectic diffeomorphism. 

Recall that the image $\phi_1(L)$ is contained in $\mu^{-1}(c)=S^{2n-1}_r$ for some $r\in (0,\infty)$ (see Section 2). 
On the other hand, we see the restriction map
\begin{align*}
\Phi|_{S^{2n-1}_r}: S^{2n-1}_r=S^{2n-1}(\tanh r)\xrightarrow{\sim} S^{2n-1}(\sinh r),\quad \Phi(z)=\cosh r\cdot z
\end{align*}
is a diffeomorphism, and $\phi_2(L)$ is contained in $S^{2n-1}(\sinh r)$. Namely, we have the following diagram:
\begin{align*}
  \begin{picture}(65, 70)(0,0)
  \put(9, 60){\hspace{-0.5cm}$(B^{n},\omega) \xrightarrow[\textup{symp. diffeo.}]{\Phi} (\mathbb{C}^n, \omega_0)$}
  \put(10, 45){$\cup \hspace{3.2cm} \cup$}
  \put(-50, 30){$L\quad \rightarrow\quad S^{2n-1}_r\hspace{0.2cm} \xrightarrow[\textup{diffeo.}]{\Phi}\hspace{0.2cm} S^{2n-1}(\sinh r)$}
  \put(-50,15){$\downarrow\hspace{1.4cm} \pi_1\downarrow$ \hspace{2.4cm} $\pi_2\downarrow$}
  \put(-60,0){$L/Z\ \rightarrow\quad \mathbb{C}P^{n-1}(\frac{4}{\sinh^2 r})\  =\  \mathbb{C}P^{n-1}(\frac{4}{\sinh^2 r})$}
  \end{picture}
\end{align*}
Here,  $\pi_1$ and $\pi_2$ are natural projections by the $Z$-actions on ${S^{2n-1}_r}$ and  ${S^{2n-1}}(\sinh r)$, respectively.
Note that  we have isomorphisms
\begin{align}\label{iso}
\Phi_*|_{E_z}: E_z\xrightarrow{\sim} E_{\Phi(z)}\quad {\rm and}\quad \Phi_*|_{\mathfrak{z}_z}: \mathfrak{z}_z\xrightarrow{\sim}\mathfrak{z}_{\Phi(z)}
\end{align}
for any $z\in \mu^{-1}(c)=S^{2n-1}_r$ since $\Phi$ is $C(K)$-equivariant, where $E$ and $\mathfrak{z}$ are  defined by \eqref{dec}.  Moreover, the Reeb vector fields and  the inner unit normal vector fields of the hypersurfaces $S^{2n-1}_r\subset B^{n}$ and $S^{2n-1}(\sinh r)\subset \mathbb{C}^n$ are given by
 \begin{align}\label{xi}
&\xi_1(z):=\frac{\tilde{v}_z}{|\tilde{v}_z|_g}=\frac{\tilde{v}_z}{\sinh r\cosh r}\quad {\rm and}\quad \xi_2(\Phi(z)):=\frac{\tilde{v}_{\Phi(z)}}{|\tilde{v}_{\Phi(z)}|_{g_{0}}}=\frac{\Phi_*\tilde{v}_{z}}{\sinh r},\\
&N_1:=J\xi_1 \quad {\rm and}\quad N_2:=J_{0}\xi_2,\nonumber
\end{align}
respectively. Also, we define 1-forms $\eta_1:=g(\xi_1,\cdot)|_{S^{2n-1}_r}$ and $\eta_2:=g_{0}(\xi_2,\cdot)|_{S^{2n-1}(\sinh r)}$.

Let us  consider the induced metrics on $L$ 
\begin{align*}
g_1:=\phi_1^*g\quad {\rm and}\quad g_2:=\phi_2^*g_0=(\Phi\circ \phi_1)^*g_0.
\end{align*}
For any point $p\in L$, we have decompositions
\begin{align*}
T_{\phi_{\alpha}(p)}L=E_{\phi_{\alpha}(p)}^l\oplus \mathfrak{z}_{\phi_{\alpha}(p)},\quad {\rm where}\quad \mathfrak{z}_{\phi_{\alpha}(p)}:= {\rm span}_{\mathbb{R}}\{\tilde{v}_{\phi_{\alpha}(p)}\}
\end{align*}
for $\alpha=1,2$. This is an orthogonal decomposition with respect to the metric $g_{\alpha}$. By \eqref{iso}, we have isomorphsims $\Phi_*|_{E_{\phi_1(p)}^l}: E_{\phi_1(p)}^l\xrightarrow{\sim} E_{\phi_2(p)}^l $ and $\Phi_*|_{\mathfrak{z}_{\phi_1(p)}}: \mathfrak{z}_{\phi_1(p)}\xrightarrow{\sim} \mathfrak{z}_{\phi_2(p)}$.  Because of this reason, we simply write
\begin{align*}
T_pL=E^l_p\oplus \mathfrak{z}_p
\end{align*}
and use identifications $E_p^l\simeq E^l_{\phi_1(p)}\simeq E^l_{\phi_2(p)}$ and $\mathfrak{z}_p\simeq \mathfrak{z}_{\phi_1(p)}\simeq \mathfrak{z}_{\phi_2(p)}$ in the following.
According to this decomposition (with identifications via $\Phi$), it turns out that the induced metrics $g_1$ and $g_2$ are decomposed into
\begin{align}\label{deco2}
g_1=g_E\oplus (\cosh^2r\cdot g_{\mathfrak{z}})\quad {\rm and}\quad g_2=g_E\oplus g_{\mathfrak{z}},
\end{align}
respectively, where $g_E:=g_2|_{E_p^l}$ and $g_{\mathfrak{z}}:=g_2|_{\mathfrak{z}_p}$. In fact, for $\alpha=1,2$, we have $g_{\alpha}|_{E_p^l}=\pi_{\alpha}^*(\phi_c^*g_{FS})$, where $\phi_c: L/Z\to \mathbb{C}P^{n-1}(4/\sinh^2r)$ and $g_{FS}$ is the Fubini-Study metric on $\mathbb{C}P^{n-1}(4/\sinh^2r)$. On the other hand, we have
\begin{align*}
\xi_1=\frac{1}{\cosh r}\xi_2
\end{align*}
by \eqref{xi}, and this implies $g_1|_{\mathfrak{z}_p}=\cosh^2r\cdot g_2|_{\mathfrak{z}_p}$ as given in \eqref{deco2}. In particular, we can take local orthonormal bases of $L$ with respect to $g_1$ and $g_2$ by $\{e_1,\ldots, e_{n-1}, \xi_1\}$ and $\{e_1,\ldots, e_{n-1}, \xi_2\}$, respectively, where $\{e_i\}_{i=1}^{n-1}$ is an orthonormal basis of $(E_p^l, g_E)$. In other words, we take $\{e_i\}_{i=1}^{n-1}$ so that $\{\overline{e}_i:=(\pi_1\circ \phi_1)_*e_i\}_{i=1}^{n-1}$ is a local orthonormal basis of $L/Z$ in $\mathbb{C}P^{n-1}(4/\sinh^2 r)$.

Denote the norm, the Levi-Civita connection, gradient and Hodge-de Rham Lapacian for function $u\in C^{\infty}(L)$ with respect to $g_1:=\phi_1^*g$ and $g_2:=\phi^*_2g_{0}$ by $|\cdot|_1$ and $|\cdot|_2$, $\nabla^1$ and $\nabla^2$, $\nabla_1u$ and $\nabla_2u$, $\Delta_1u$ and $\Delta_2u$, respectively.  

\begin{lemma}\label{c1}
We have the following:
\begin{itemize}
\item[(a)] For any $X\in T_pL$, we have $|X|_1^2=|X|_2^2+\sinh^2r \cdot \eta_2(X)^2$.
\item[(b)] For any $u\in C^{\infty}(L)$, we have $\nabla_1u=\nabla_2 u-\tanh^2r\cdot \xi_2(u)\xi_2$. Moreover, 
 \begin{align*}
 |\nabla_1u|^2_1=|\nabla_2u|_2^2-\tanh^2r\cdot \xi_2(u)^2.
 \end{align*}
\item[(c)] Let $\{e_1,\ldots, e_{n-1}, \xi_1\}$ and $\{e_1,\ldots, e_{n-1},\xi_2\}$ be the local frame of $L$ taking above. Then, the Levi-Civita connections are related as follows: 
\begin{align*}
\nabla^1_{e_i}e_j&=\nabla^2_{e_i}e_j\quad {\rm for}\quad i,j=1,\ldots, n-1.\\
\nabla^\alpha_{e_i}\xi_{\alpha}&=\nabla^\alpha_{\xi_{\alpha}}e_i=\nabla^{\alpha}_{\xi_\alpha}\xi_{\alpha}=0\quad {\rm for}\quad \alpha=1,2,\quad {i=1,\ldots, n-1}.
\end{align*} 
\item[(d)] For any $u\in C^{\infty}(L)$, we have $\Delta_1u=\Delta_2u+\tanh^2r\cdot\xi_2(\xi_2u)$.
\end{itemize}
\end{lemma}

\begin{proof}
For any $X\in T_pL$, we  set $X=\sum_{i=1}^{n-1}X_ie_i+X_n\xi_1=\sum_{i=1}^{n-1}X_ie_i+\frac{1}{\cosh r}X_n\xi_2$.  Then, we see
\begin{align*}
|X|_1^2=\sum_{i=1}^{n-1}X_i^2+X_n^2=|X|_2^2+\Big(1-\frac{1}{\cosh^2 r}\Big)X_n^2=|X|_2^2+\sinh^2 r\cdot\eta_2(X)^2.
\end{align*}
This proves (a). Next, we see
\begin{align*}
\nabla_1u&=\sum_{i=1}^{n-1}(e_iu)e_i+(\xi_1u)\xi_1=\sum_{i=1}^{n-1}(e_iu)e_i+\frac{1}{\cosh^2 r}(\xi_2u)\xi_2
=\nabla_2u-\Big(1-\frac{1}{\cosh^2r}\Big)(\xi_2u)\xi_2\\
&=\nabla_2u-\tanh^2r\cdot (\xi_2u)\xi_2.
\end{align*}
Moreover, by using (a), we have
\begin{align*}
|\nabla_1u|^2_1&=|\nabla_1u|^2_2+\sinh^2 r\cdot \eta_2(\nabla_1u)^2\\
&=\Big|\nabla_2u-\tanh^2r\cdot \xi_2(u)\xi_2\Big|^2_2+\sinh^2r\cdot\eta_2\Big(\nabla_2u-\tanh^2r\cdot \xi_2(u)\xi_2\Big)^2\\
&=|\nabla_2u|_2^2-2\tanh^2r\cdot\xi_2(u)^2+\tanh^4 r\cdot \xi_2(u)^2+\sinh^2r\cdot \{(1-\tanh^2r)\xi_2(u)\}^2 \\
&=|\nabla_2u|_2^2-\tanh^2r\cdot \xi_2(u)^2,
\end{align*}
where we used a relation $\sinh^2r=\tanh^2 r/(1-\tanh^2 r)$.
 This proves (b).
 
 Next, we shall show (c).  Since $\pi_1, \pi_2: L\rightarrow L/Z$ are  Riemannian submersions, we have 
 \begin{align*}
{\nabla}^c_{\overline{e}_i}\overline{e}_j= (\pi_1\circ \phi_1)_*(\nabla^1_{e_i}e_j)&=(\pi_2\circ \phi_2)_*(\nabla^2_{e_i}e_j),
 \end{align*}
 where  $\nabla^c$  is the Levi-Civita connection on $L/Z$. This implies
 $(\nabla^1_{e_i}e_j)^{\top_{E^l}^{1}}=(\nabla^2_{e_i}e_j)^{\top_{E^l}^2}$, where $\top_{E^l}^{\alpha}$ means the orthogonal projection with respect to $g_{\alpha}$ onto $E^l$.
  On the other hand, we see
 \begin{align}\label{ija}
 g_{\alpha}(\nabla^{\alpha}_{e_i}e_j, \xi_{\alpha})&=-g_{\alpha}(\overline{\nabla}^{\alpha}_{e_i}e_j, JN_{\alpha})=-g_{\alpha}(\overline{\nabla}^{\alpha}_{e_i}N_{\alpha}, Je_j)=g_{\alpha}(\widetilde{A}^{\alpha}(e_i), Je_j)=0
 \end{align}
 since $Je_j\in E$ and the $\eta$-umbilical conditions. Therefore, we have $\nabla^{\alpha}_{e_i}e_j=(\nabla^{\alpha}_{e_i}e_j)^{\top_{E^l}}$, and we obtain $\nabla^1_{e_i}e_j=\nabla^2_{e_i}e_j$.
 
 Next, we consider $\nabla^{\alpha}_{e_i}\xi_{\alpha}$ and $\nabla^{\alpha}_{\xi_{\alpha}}\xi_{\alpha}$. Since $|\xi_{\alpha}|_{g_{\alpha}}=1$, we have $g_{\alpha}(\nabla^{\alpha}_{e_i}\xi_{\alpha}, \xi_{\alpha})=0$. Moreover, \eqref{ija} shows $g_{\alpha}(\nabla^{\alpha}_{e_i}\xi_{\alpha}, e_j)=-g_{\alpha}(\nabla^{\alpha}_{e_i}e_j,\xi_{\alpha})=0$ for any $i,j=1,\ldots, n-1$. Thus, $\nabla^{\alpha}_{e_i}\xi_{\alpha}=0$. Moreover, since $e_i$ is $C(K)$-invariant, we have $[\tilde{v}, e_i]=0$, and hence, $\nabla^{\alpha}_{\xi_{\alpha}}e_i=(const.)\nabla^{\alpha}_{\tilde{v}}e_i=(const.)\nabla^{\alpha}_{e_i}\tilde{v}=\nabla^{\alpha}_{e_i}{\xi}_{\alpha}=0$, where $const.$ is depends only on $\alpha$ and $r$. Similarly, we have $g_{\alpha}(\nabla^{\alpha}_{\xi_{\alpha}}\xi_{\alpha}, \xi_{\alpha})=0$, and 
 \begin{align*}
 g_{\alpha}(\nabla^{\alpha}_{\xi_{\alpha}}\xi_{\alpha}, e_i)=-g_{\alpha}(\nabla^{\alpha}_{\xi_{\alpha}}e_i,\xi_{\alpha})=-(const.) g_{\alpha}(\nabla^{\alpha}_{\tilde{v}}e_i, \tilde{v})=-(const.) g_{\alpha}(\nabla^{\alpha}_{e_i}{\tilde{v}}, \tilde{v})=0
 \end{align*}
 since $|\tilde{v}|_{g_{\alpha}}$ is constant on $L$.
   Thus, we obtain $\nabla^{\alpha}_{\xi_{\alpha}}\xi_{\alpha}=0$. This proves (c).
   
   Finally, we show (d). In the local orthonormal frame, by using (c), we see
   \begin{align*}
   -\Delta_1u&=\sum_{i=1}^{n-1}e_i(e_iu)+\xi_1(\xi_1u)-\sum_{i=1}^{n-1}(\nabla^{1}_{e_i}e_i)u\\
   &=\sum_{i=1}^{n-1}e_i(e_iu)+\frac{1}{\cosh^2 r}\xi_2(\xi_2u)-\sum_{i=1}^{n-1}(\nabla^{2}_{e_i}e_i)u\\
   &=-\Delta_2u-\tanh^2r\cdot \xi_2(\xi_2u).
   \end{align*}
   This proves (d).
\end{proof}

Next, we compare extrinsic properties. Denote the second fundamental form and the mean curvature vector of the immersion $\phi_1: L\rightarrow (B^{n},g)$ and $\phi_2: L\rightarrow (\mathbb{C}^n, g_{0})$ by $B_1$ and $B_2$, $H_1$ and $H_2$, respectively. Also, we set $H_{\alpha}':=(H_{\alpha})_E$ and  $S_{\alpha}(X,Y,W):=g_{\alpha}(B_{\alpha}(X,Y), J_{\alpha}W)$
for $X,Y,W\in \Gamma(TL)$ and $\alpha=1,2$ as introduced in Section \ref{sec2}, where $J_1$ and $J_2$ denotes the complex structure on $B^n$ and $\mathbb{C}^n$, respectively.

\begin{lemma}\label{c2}
For $X,Y,W\in E_p^l$, we have 
\begin{align}
\label{c21}  S_1(X,Y,W)=S_2(X,Y,W).
\end{align}
In particular, we see
\begin{align}\label{c22}
g_1(J_1H_1', W)=g_2(J_2H_2', W)\quad {\rm and}\quad S_1(X,Y, J_1H_1')=S_2(X,Y, J_2H_2').
\end{align}
\end{lemma}
\begin{proof}
For $i,j,k=1,\ldots, n-1$ and $\alpha=1,2$, we have
\begin{align*}
S_{\alpha}(e_i,e_j,e_k)=g_{\alpha}(\overline{\nabla}^{\alpha}_{e_i}e_j, J_{\alpha}e_k)=g_{\alpha}(\widetilde{\nabla}^{\alpha}_{e_i}e_j, J_{\alpha}e_k)=g_c(\overline{\nabla}^c_{\overline{e}_i}\overline{e}_j, J_c \overline{e}_k)
\end{align*}
since $\pi_{\alpha}$ is a  Riemannian submersion onto $\mathbb{C}P^{n-1}(4/\sinh^2r )$ and $(\pi_{\alpha})_*\circ J_{\alpha}=J_c\circ (\pi_{\alpha})_*$. This shows  $S_{1}(e_i,e_j,e_k)=S_{2}(e_i,e_j,e_k)$ for any $i,j,k=1,\ldots, n-1$, and hence, we obtain  \eqref{c21}. 

Recall that  $S_1(\xi_1,\xi_1, W)=S_2(\xi_2,\xi_2, W)=0$ for $W\in E_p^l$ (see the proof of Lemma \ref{sform}), and hence, by taking the trace of the former two components of $S$ and using the fact that $E_p^l$ is $J_{\alpha}$-invariant,  we obtain the first equality of \eqref{c22}.  Moreover, we have
\begin{align*}
S_{\alpha}(e_i,e_j,J_{\alpha}H_{\alpha}')=-g_c(\overline{\nabla}^c_{\overline{e}_i}\overline{e}_j, (\pi_{\alpha})_*H_{\alpha}').
\end{align*}
Here, it turns out that $(\pi_{\alpha})_*H_{\alpha}'$ coincides with the mean curvature vector $H_c$ of the reduced Lagrangian immersion $L/Z\to \mathbb{C}P^{n-1}(4/\sinh^2r)$. This can be shown by using Lemma 3 in \cite{K}  and the fact that,  in our setting,  $|\tilde{v}_z|_{g_{\alpha}}$ is constant on $L$ for each $\alpha$. This proves the second equation of \eqref{c22}.
\end{proof}

On the other hand,  the shape operators $\widetilde{A}^1$ of $S^{2n-1}_r\rightarrow B^{2n}$ and $\widetilde{A}^2$ of $S^{2n-1}(\sinh r)\rightarrow \mathbb{C}^n$ (with respect to $N_1:=J\xi_1$ and $N_2:=J_0\xi_2$, respectively) satisfy 
\begin{align}\label{shape}
\widetilde{A}^1(X)=\coth r\cdot X+\tanh r\cdot \eta_1(X)\xi_1\quad {\rm and}\quad \widetilde{A}^2(Y)=\frac{1}{\sinh r}Y \end{align}
for $X\in \Gamma(TS^{2n-1}_r)$ and $Y\in \Gamma(TS^{2n-1}(\sinh r))$, respectively.

\begin{lemma}\label{c3}
We have
\begin{align*}
S_1(\nabla_1u, \nabla_1u, J_1H_1)&=S_2(\nabla_2u, \nabla_2u, J_2H_2)\\
&\quad -(n+1) |\nabla_2u|^2_2+(n+\tanh^2r)\tanh^2r\xi_2(u)^2\quad {\rm and}\\
J_1H_1(u)&=J_2H_2(u)-\frac{\tanh r}{\cosh r}\xi_2(u).
\end{align*}
\end{lemma}
\begin{proof}
By \eqref{shape} and Lemma \ref{sform}, we have
\begin{align}
\label{c31}
S_1(\nabla_1u, \nabla_1u, J_1H_1)
&=S_1\Big((\nabla_1u)_E,(\nabla_1u)_E, J_1H_1'\Big)+2\coth r\cdot \xi_1(u)g_1\Big(J_1H_1' ,(\nabla_1 u)_E\Big)\\
&\quad -(n\coth r +\tanh r){\{}\coth r |\nabla_1u|^2_1+\tanh r\xi_1(u)^2\}.\nonumber\\
\label{c32}
S_2(\nabla_2u, \nabla_2u, J_2H_2)
&=S_2\Big((\nabla_2u)_E,(\nabla_2u)_E, J_2H_2'\Big)+\frac{2}{\sinh r}\cdot \xi_2(u)g_2\Big(J_2H_2' ,(\nabla_2 u)_E\Big)\\
&\quad -\frac{n}{\sinh^2 r} |\nabla_2u|^2_2.\nonumber
\end{align}
Here, by Lemma \ref{c2} and the relation $\xi_1=\frac{1}{\cosh r}\xi_2$, it turns out that the first two terms in the RHS of \eqref{c31} and \eqref{c32} coincides with each other. Therefore, by using Lemma \ref{c1} we see
\begin{align*}
&S_1(\nabla_1u, \nabla_1u, J_1H_1)-S_2(\nabla_2u, \nabla_2u, J_2H_2)\\
&=-(n\coth r +\tanh r)\Big{\{}\coth r \Big(|\nabla_2u|^2_2-\tanh^2r\xi_2(u)^2\Big)+\tanh r\frac{\xi_2(u)^2}{\cosh^2r}\Big{\}}+\frac{n}{\sinh^2 r} |\nabla_2u|^2_2\\
&=-(n\coth r +\tanh r)\Big{\{}\coth r\cdot |\nabla_2u|^2_2-\tanh^3r\cdot \xi_2(u)^2\Big{\}}+\frac{n}{\sinh^2 r} |\nabla_2u|^2_2\\
&=-(n+1) |\nabla_2u|^2_2+(n+\tanh^2r)\tanh^2r\xi_2(u)^2.
\end{align*}

On the other hand, we see
\begin{align*}
J_1H_1(u)&=J_1H_1'(u)-(n\coth r+\tanh r)\xi_1(u)\\
&=J_2H_2'(u)-\Big(\frac{n}{\sinh r}+\frac{\tanh r}{\cosh r}\Big)\xi_2(u)=J_2H_2(u)-\frac{\tanh r}{\cosh r}\xi_2(u)
\end{align*}
by \eqref{mcl}, \eqref{shape} and \eqref{c22}.
\end{proof}

Now, we are ready to prove the following formula:

\begin{theorem}\label{keyprop}
Let  $\phi: L\rightarrow \mathbb{C}H^n(-4)$ be a $C(K)$-invariant Lagrangian embedding whose image is contained in the geodesic hypersphere $S^{2n-1}_r\subset \mathbb{C}H^n(-4)$ of geodesic radius $r\in (0,\infty)$. Suppose furthermore $\phi$ is H-minimal in $\mathbb{C}H^n(-4)$. Then, $\phi$ is H-stable in $\mathbb{C}H^n(-4)$ if and only if the corresponding Lagrangian embedding $\phi_2:=\Phi\circ \phi: L\rightarrow \mathbb{C}^n$   satisfies
\begin{align} \label{svfre}
&\int_L |\Delta_2u|^2-2g_2(B_2(\nabla_2u, \nabla_2u), H_2)+J_2 H_2(u)^2\\
&\quad+2\tanh^2r\cdot \Delta_2u\cdot \xi_2\xi_2(u)-2\frac{\tanh r}{\cosh r}\xi_2(u)J_2H_2(u)\nonumber\\
&\quad+\tanh^4r|\xi_2\xi_2(u)|^2-\frac{\tanh^2 r}{\cosh^2r}|\xi_2(u)|^2dv_{g_2}\geq 0\nonumber
\end{align}
for any $u\in C^{\infty}(L)$, where $\xi_2$ is the Reeb vector field on the hypersphere $S^{2n-1}$ containing $\phi_2(L)$ in $\mathbb{C}^n$ so that $N_2:=J_{2}\xi_2$ is the inner unit normal vector field of $S^{2n-1}$.
\end{theorem}

\begin{proof}
For  a function $u\in C^{\infty}(L)$, let $\phi_{1,s}$ and $\phi_{2,s}$ be a Hamiltonian deformation of $\phi_1: L\rightarrow S^{2n-1}_r\subset B^{n}$ and $\phi_2: L\rightarrow S^{2n-1}(\sinh r)\subset \mathbb{C}^n$ so that $d/ds|_{s=0}\phi_{\alpha,s}=J_{\alpha}\nabla_{\alpha}u$ for $\alpha=1,2$. We denote the integrand of the right hand side of the second variational formula \eqref{svf}  for $\phi_{\alpha,s}$  by $\mathcal{J}_{\alpha}(u)$. By Lemma \ref{c1} and \ref{c3},  we have
\begin{align}\label{c4}
\mathcal{J}_1(u)&=|\Delta_1u|^2+2(n+1)|\nabla_1u|^2_1+2S_1(\nabla_1u, \nabla_1u, J_1H_1)+J_1H_1(u)^2\\
&=|\Delta_2u+\tanh^2r \xi_2\xi_2(u)|^2+2(n+1)(|\nabla_2u|_2^2-\tanh^2r |\xi_2(u)|^2)\nonumber\\
&\quad +2S_2(\nabla_2u, \nabla_2u, J_2H_2)-2(n+1)|\nabla_2u|_2^2+2(n+\tanh^2r)\tanh^2r|\xi_2(u)|^2\nonumber\\
&\quad +\Big|J_2H_2(u)-\frac{\tanh r}{\cosh r}\xi_2(u)\Big{|}^2\nonumber\\
&=|\Delta_2u|^2+2S_2(\nabla_2u, \nabla_2u, J_2H_2)+J_2H_2(u)^2\nonumber\\
&\quad +2\tanh^2r\cdot \Delta_2u\cdot \xi_2\xi_2(u)-2\frac{\tanh r}{\cosh r}\xi_2(u)J_2H_2(u)\nonumber\\
&\quad +\tanh^4r|\xi_2\xi_2(u)|^2-\frac{\tanh^2r}{\cosh^2r}|\xi_2(u)|^2.\nonumber
\end{align}
On the other hand, one easily checked that the volume measure has a relation $dv_{g_1}=\cosh r\cdot dv_{g_2}$. Therefore, by integrating \eqref{c4} over $L$ by $dv_{g_1}$, we obtain the conclusion.
\end{proof}

We remark that the $C(K)$-invariant Lagrangian submanifold $L$ in $\mathbb{C}H^n$ is H-minimal if and only if so is the reduced Lagrangian submanifold $L/Z$ in $\mathbb{C}P^{n-1}$ (cf. \cite{Dong}). Moreover, a typical examples of H-minimal Lagrangian is obtained by a {\it compact homogeneous} Lagrangian submanifold in $\mathbb{C}H^n$, i.e. a Lagrangian orbit of $K'$-action for a connected compact subgroup $K'\subset K$.  Since $\Phi: B^{n}\rightarrow \mathbb{C}^n$ is a $K$-equivariant symplectic diffeomorphism, it turns out that any compact homogeneous Lagrangian submanifold in $\mathbb{C}H^n$ corresponds to a compact homogeneous Lagrangian submanifold in $\mathbb{C}^n$ (See Theorem 1 in \cite{HK}).  Theorem \ref{keyprop} is applicable to all such examples.

\section{The torus orbits in $\mathbb{C}H^n$}\label{sec4}
In this section, we consider the Hamiltonian stability of torus orbits in $\mathbb{C}H^n(-4)$, and give a proof of Theorem \ref{maintheorem}.  Let $T^n$ be a maximal torus of $K=S(U(1)\times U(n))$  represented by
\begin{align*}
T^n:=\{{\rm diag}(e^{-\sqrt{-1}\sum_{i=1}^n\theta_i},e^{\sqrt{-1}\theta_1},\ldots, e^{\sqrt{-1}\theta_n}); \theta_i \in\mathbb{R}\ \forall i=1,\ldots, n\}
\end{align*}
Since $\Phi$ is $K$-equivariant, it is easy to see that any $T^n$-orbit $T^n\cdot z$ through $z\in B^n\simeq  \mathbb{C}H^n$ corresponds to a standard $T^n$-orbit in $\mathbb{C}^n$ via the map $\Phi$:
\begin{align}\label{tori}
\Phi(T^n\cdot z)=T(r_1,\ldots, r_n):=\{(r_1e^{\sqrt{-1}\theta_1},\ldots r_ne^{\sqrt{-1}\theta_n});\theta_i\in \mathbb{R}\}\subset \mathbb{C}^n.
\end{align}
for some $(r_1,\ldots, r_n)\in (\mathbb{R}_{>0})^n$. Note that this correspondence is one to one. In particular, any  $T^n$-orbit in $\mathbb{C}H^n$ is Lagrangian since so is $T^n$-orbit in $\mathbb{C}^n$. Moreover, they are all H-minimal. Thus, by Theorem \ref{keyprop}, we consider a principal $T^n$-orbit in $\mathbb{C}^n$ in order to show the H-stability of $T^n$-orbit in $\mathbb{C}H^n$.

\subsection{Hamiltonian stability of torus orbits}
Let $S^{2n-1}(\sinh r)$ be the hypersphere of radius $\sinh r$ for $r\in(0,\infty)$ in $\mathbb{C}^n$. The Reeb vector field on $S^{2n-1}(\sinh r)$ is given by
\begin{align*}
\xi_2:=\frac{1}{\sinh r}\sum_{i=1}^n \partial_i,
\end{align*}
where $\partial_i$ is a tangent vector field on $S^{2n-1}$ defined by
\begin{align*}
\partial_i(z):=-y^i\frac{\partial}{\partial x^i}+x^i\frac{\partial}{\partial y^i}
\end{align*}
for $i=1,\ldots, n$, where $z^i=x^i+\sqrt{-1}y^i$.  
Note that $N:=J_0\xi=-{\bf p}$ is the inner unit normal vector field on $S^{2n-1}(\sinh r)$.

Let us consider the standard $T^n$-action on $\mathbb{C}^n$ so that the principal orbit is a Lagrangian torus given by \eqref{tori}.
A moment map $\mu: \mathbb{C}^n\rightarrow \mathbb{R}^n$ of the $T^n$-action on $\mathbb{C}^n$ is given by $\mu(z):=(-\frac{1}{2}|z_1|^2, \ldots, -\frac{1}{2}|z_n|^2)$ and  we identify the moment polytope $\mu(\mathbb{C}^n)$ with a quadrant
\begin{align*}
P:=\{(p_1,\ldots p_n)\in \mathbb{R}^n; p_i\geq 0\}, \quad \mu(z)\mapsto (|z_1|^2,\ldots, |z_n|^2).
\end{align*}
It is easy to see that the map $\mu$ gives rise to a one to one correspondence between principal $T^n$-orbits and the set of interior points $P^{int}$ of $P$. For each $r\in (0,\infty)$, we denote the set of torus orbits contained in $S^{2n-1}(\sinh r)$ by
\begin{align*}
\mathcal{O}_r:=\Big{\{}T(r_1,\ldots, r_n);\ \sum_{i=1}^nr_i^2=\sinh^2 r\Big{\}}
\end{align*}
By using the correspondence via the moment map, we have a correspondence
\begin{align*}
\mathcal{O}_r\xrightarrow{1:1} \Pi_r:=\Big{\{}(p_1,\ldots, p_n)\in P^{int};\ \sum_{i=1}^np_i=\sinh^2r\Big{\}}.
\end{align*}
Moreover, we parametrize $\mathcal{O}_r$ by
\begin{align}\label{param}
&\Big{\{}\widetilde{\bf s}:=(s_1,\ldots, s_n)\in (\mathbb{R}_{>0})^n;\ \sum_{i=1}^ns_i=1\Big{\}}=\Pi_{\sinh^{-1}(1)}\xrightarrow{\sim} \mathcal{O}_r,\\
&(s_1,\ldots s_n)\quad\mapsto\quad T_{r, \widetilde{\bf s}}^n:=S^1(\sinh r\sqrt{s_1})\times \cdots \times S^1(\sinh r\sqrt{s_n}).\nonumber
\end{align}

 We take a basis of $T_zT^n_{r, \widetilde{\bf s}}$ by
\begin{align*}
\frac{\partial}{\partial \theta_i}\Big{|}_z=-y^i\frac{\partial}{\partial x^i}\Big{|}_z+x^i\frac{\partial}{\partial y^i}\Big{|}_z=\partial_i|_z
\end{align*}
for $i=1,\ldots, n$ and $z\in T^n_{r, \widetilde{\bf s}}$. Note that we have $g(\partial_i, \partial_j)=(\sinh^2r\cdot s_i)\delta_{ij}$.
Then, one easily computes  the second fundamental form and the mean curvature vector of $T^n_{r, \widetilde{\bf s}}$ in $\mathbb{C}^n$ as follows:
\begin{align}\label{bh}
B_2(\partial_i, \partial_j)=\delta_{ij}J\partial_i\quad {\rm and}\quad H_2=\sum_{i=1}^n \frac{1}{\sinh^2r\cdot s_i}J\partial_i,
\end{align}
respectively.

\begin{lemma}\label{le1}
The torus orbit $\Phi^{-1}(T^n_{r, \widetilde{\bf s}})$ is Hamiltonian-stable in $\mathbb{C}H^n(-4)$ if and only if
\begin{align}\label{stab}
Q_{n,r}(\widetilde{\bf s},{\bf m}):=a_1(\widetilde{\bf s},{\bf m})-2\tanh^2r\cdot a_2(\widetilde{\bf s},{\bf m})+\tanh^4r\cdot a_3({\bf m})\geq 0
\end{align}
 for any ${\bf m}=(m_1,\ldots, m_n)\in \mathbb{Z}^n\setminus \{\bf 0\}$, where
\begin{align*}
a_1(\widetilde{\bf s},{\bf m})&:=\Big(\sum_{i=1}^n\frac{m_i^2}{s_i}\Big)^2+\Big(\sum_{i=1}^n\frac{m_i}{s_i}\Big)^2-2\Big(\sum_{i=1}^n\frac{m_i^2}{s_i^2}\Big)=\sum_{i=1}^n\frac{m_i^2(m_i^2-1)}{s_i^2}+\sum_{i\neq j}\frac{m_im_j(m_im_j+1)}{s_is_j},\\
a_2(\widetilde{\bf s},{\bf m})&:=\Big(\sum_{i=1}^nm_i\Big)\Big{\{}\Big(\sum_{i=1}^n \frac{m_i^2}{s_i}\Big)\Big(\sum_{i=1}^nm_i\Big)-\Big(\sum_{i=1}^n \frac{m_i}{s_i}\Big)\Big{\}},\nonumber\\
a_3({\bf m})&:=\Big(\sum_{i=1}^n m_i\Big)^2\Big{\{}\Big(\sum_{i=1}^n m_i\Big)^2-1\Big{\}}.\nonumber
\end{align*}
 \end{lemma}

\begin{proof}
We set $r_i:=\sinh r\cdot \sqrt{s_i}$ so that $T^n_{r,\widetilde{\bf s}}=S^{1}(r_1)\times\cdots\times S^1(r_n)$ in this proof.

We decompose the integrand of the formula \eqref{svfre} into three parts
\begin{align*}
&\underbrace{|\Delta_2 u|^2-2g_2(B_2(\nabla_2u, \nabla_2u), H_2)+ J_2 H_2(u)^2}_{\rm (I)}\\
&\quad +\underbrace{\Big{\{}c_1(r)\cdot \xi_2\xi_2(u)\cdot \Delta_2u-c_2(r)\cdot\xi_2(u)J_2H_2(u)\Big{\}}}_{\rm (I\hspace{-1pt}I)}+\underbrace{\Big{\{}c_3(r)|\xi_2\xi_2(u)|^2-c_4(r)\xi_2(u)^2\Big{\}}}_{\rm(I\hspace{-1pt}I\hspace{-1pt}I)},
\end{align*}
 where we set
\begin{align*}
c_1(r):&=2\tanh^2r,\quad c_2(r):=\frac{2\tanh r}{\cosh r},\\
c_3(r):&=\tanh^4r,\quad c_4(r):=\frac{\tanh^2r}{\cosh^2r}.
\end{align*}

Since the integral of ${\rm(I)}$ coincides with the second variation of $T^n_{r,\widetilde{\bf s}}$ in $\mathbb{C}^n$, the same calculation given in \cite{Oh2} (see (29) in \cite{Oh2}) shows that 
\begin{align}\label{s1}
\int_{T^n_{r,\widetilde{\bf s}}}{\rm (I)}dv_{g_2}=\int_{T^n_{r,\widetilde{\bf s}}}\Big{\{}\sum_{i=1}^n\frac{1}{r_i^4}(\partial_i^4u+\partial_i^2u)+\sum_{i\neq j}\frac{1}{r_i^2r_j^2}(\partial_i^2\partial_j^2u-\partial_i\partial_ju)\Big{\}}udv_{g_2}.
\end{align}
Next, we calculate (I\hspace{-1pt}I). A straightforward calculation shows that
\begin{align}\label{s2}
\int_{T^n_{r,\widetilde{\bf s}}}{\rm (I\hspace{-1pt}I)}dv_{g_2}
&=\int_{T^n_{r,\widetilde{\bf s}}} \frac{c_1(r)}{\sinh^2r}\Big(\sum_{i,j=1}^n \partial_i\partial_ju\Big)\Big(-\sum_{k=1}^n\frac{1}{r_k^2}\partial_k^2u\Big)-\frac{c_2(r)}{\sinh r}\Big(\sum_{i=1}^n\partial_iu\Big)\Big(-\sum_{k=1}^n\frac{1}{r_k^2}\partial_k u\Big)dv_{g_2}\\
&=-\frac{2}{\cosh^2r}\int_{T^n_{r,\widetilde{\bf s}}} \Big{\{}\sum_{i,j,k=1}^n \frac{1}{r_k^2}\partial_k^2\partial_i\partial_ju+\sum_{i,k=1}^n\frac{1}{r_k^2}\partial_k\partial_iu\Big{\}}udv_{g_2} \nonumber
\end{align}
Here, we used the integration by parts. Finally, we see
\begin{align}\label{s3}
\int_{T^n_{r,\widetilde{\bf s}}}{\rm (I\hspace{-1pt}I\hspace{-1pt}I)}dv_{g_2}
&=\int_{T^n_{r,\widetilde{\bf s}}} \frac{c_3(r)}{\sinh^4r}\Big|\sum_{i,j=1}^n \partial_i\partial_j u\Big|^2-\frac{c_4(r)}{\sinh^2 r}\Big(\sum_{i=1}^n \partial_iu\Big)^2dv_{g_2}\\
&=\frac{1}{\cosh^4r}\int_{T^n_{r,\widetilde{\bf s}}} \Big{\{} \sum_{i,j,k,l=1}^n \partial_k\partial_l\partial_i\partial_j u+\sum_{i,j=1}^n \partial_j\partial_iu\Big{\}}udv_{g_2}.\nonumber
\end{align}

Recall that the non-zero eigenvalues of $\Delta$ on the flat torus $T^n_{r,\widetilde{\bf s}}$ are given by
\begin{align*}
\lambda_{\bf m}=\sum_{i=1}^n \frac{m_i^2}{r_i^2}\quad {\rm for}\quad {\bf m}:=(m_1,\ldots, m_n)\in \mathbb{Z}^n\setminus \{{\bf 0}\},
\end{align*}
and the corresponding eigenspace is spanned by
\begin{align*}
u_{\bf m}^c:=\cos \Big(\sum_{i=1}^n m_i \theta_i\Big)\quad {\rm and}\quad u_{\bf m}^s:=\sin \Big(\sum_{i=1}^n m_i \theta_i\Big).
\end{align*}
It is known that these functions form  an orthogonal basis of $L^2(T^n_{r,\widetilde{\bf s}})$. Note that 
\begin{align*}
\partial_i\partial_j u_{\bf m}^c=-m_im_ju_{\bf m}^c,\quad \partial_i\partial_j \partial_k\partial_lu_{\bf m}^c&=m_im_jm_km_lu_{\bf m}^c,
\end{align*}
and $u_{\bf m}^s$ is as well. Hence, substituting $u=u_{\bf m}^c$ (or $u_{\bf m}^s$) in \eqref{s1}, \eqref{s2} and \eqref{s3}, we have
\begin{align*}
\int_{T^n_{r,\widetilde{\bf s}}} {\rm (I)}+{\rm (I\hspace{-1pt}I)}+{\rm (I\hspace{-1pt}I\hspace{-1pt}I)}dv_{g_2}
&=\int_{T^n_{r,\widetilde{\bf s}}}\Big{[}\sum_{i=1}^n\frac{1}{r_i^4}(m_i^4-m_i^2)+\sum_{i\neq j}\frac{1}{r_i^2r_j^2}(m_i^2m_j^2+m_im_j)\\
&\quad\quad-\frac{2}{\cosh^2r}\Big{\{}\sum_{i,j,k=1}^n\frac{m_k^2}{r_k^2}m_im_j-\sum_{i,k=1}^n\frac{m_k}{r_k^2}m_i\Big{\}}\\
&\quad\quad+\frac{1}{\cosh^4r}\Big{\{}\sum_{i,j,k,l=1}^n m_km_lm_im_j -\sum_{i,j=1}^n m_jm_i\Big{\}}\Big{]}(u_{\bf m}^c)^2dv_{g_2}\\
&=\frac{1}{\sinh^4r}\int_{T^n_{r, \widetilde{\bf s}}} Q_{n,r}(\widetilde{\bf s}, {\bf m})(u_{\bf m}^c)^2dv_{g_2}
\end{align*}
since  we set $r_i=\sinh r\cdot \sqrt{s_i}$, the implies the lemma. 
\end{proof}

Note that the coefficients $a_1(\widetilde{\bf s}, {\bf m})$, $a_2(\widetilde{\bf s}, {\bf m})$ and $a_3({\bf m})$ are all non-negative. We shall estimate $Q_{n,r}(\widetilde{\bf s}, {\bf m})$ in the following.

First of all, we consider a specific ${\bf m}$, namely, we suppose ${\bf m}\in \mathbb{Z}^n\setminus\{0\}$ satisfies
\begin{align*}
a_3({\bf m})=0,\ \textup{or equivalently},\ \sum_{i=1}^n m_i=0\ {\rm or}\ \pm1.
\end{align*}
We shall find a necessary condition for  the Hamiltonian stability of $T_{r, \widetilde{\bf s}}^n$ when $n\geq 3$ for such an ${\bf m}$ (Proposition \ref{p3}), which leads a proof of the first assertion of Theorem \ref{maintheorem} (b).

\begin{lemma}\label{le2}
If $\sum_{i=1}^nm_i=0$, then $Q_{n,r}(\widetilde{\bf s}, {\bf m})\geq 0$ and the equality holds if and only if there exist $i,j\in \{1,\ldots, n\}$ so that $m_i=1$, $m_j=-1$ and $m_k=0$ for other $k$.
\end{lemma}
\begin{proof}
Suppose $\sum_{i=1}^nm_i=0$. Then the latter two terms in the RHS of \eqref{stab} vanish, and hence, $Q_{n,r}(\widetilde{\bf s}, {\bf m})=a_1(\widetilde{\bf s}, {\bf m})\geq 0$ by Lemma \ref{le1}. Moreover, the equality hods if and only if $m_i^2(m_i^2-1)=0$ for any $i$ and $m_im_j(m_im_j+1)=0$ for any $i\neq j$. This is equivalent to that ${\bf m}\in \mathbb{Z}^n\setminus \{0\}$ has the form $m_i=1$, $m_j=-1$ for some $i\neq j$ and $m_k=0$ for other $k$.
\end{proof}

Next, we consider the case when $\sum_{i=1}^n m_i=\pm1$. Since $Q_{n,r}(\widetilde{\bf s}, {\bf m})=Q_{n,r}(\widetilde{\bf s}, -{\bf m})$, we may assume $\sum_{i=1}^n m_i=1$ for our purpose.  We denote such ${\bf m}$ by $\widetilde{\bf m}$. 
In this case, the last term in \eqref{stab} is vanishing and we have 
\begin{align*}
&Q_{n,r}(\widetilde{\bf s}, \widetilde{\bf m})=\sum_{i=1}^n\frac{m_i^2(m_i^2-1)}{s_i^2}+\sum_{i\neq j}\frac{m_im_j(m_im_j+1)}{s_is_j}-2\tanh^2r\cdot \sum_{i=1}^n \frac{m_i(m_i-1)}{s_i}\\
&\quad \textup{for}\ (\widetilde{\bf s}, \widetilde{\bf m})\in \Big{\{}({\bf s}, {\bf m})\in \mathbb{R}_{>0}^n\times \mathbb{Z}^n;\  \sum_{i=1}^ns_i=1,\ \sum_{i=1}^nm_i=1\Big{\}}.
\end{align*}

\begin{lemma}\label{le3}
If $m_i\neq -1$ for any $i$, then $Q_{n,r}(\widetilde{\bf s}, \widetilde{\bf m})\geq 0$ and the equality holds if and only if $\widetilde{\bf m}$ is of the form $m_i=1$ for some $i$ and $m_k= 0$ for other $k\neq i$.
\end{lemma}
\begin{proof}
$Q_{n,r}(\widetilde{\bf s}, \widetilde{\bf m})$ is rearranged as
\begin{align}\label{q12}
\sum_{i=1}^n\frac{m_i(m_i-1)\{m_i(m_i+1)-2\tanh^2r\cdot s_i\}}{s_i^2}+\sum_{i\neq j}\frac{m_im_j(m_im_j+1)}{s_is_j}.
\end{align}
 In \eqref{q12}, the second term is non-negative and the coefficient of $1/s_i^2$ is non-negative whenever  $m_i\neq -1$ since $0< \tanh^2r\cdot s_i<1$. Therefore, ${Q}_{n,r}(\widetilde{\bf s}, \widetilde{\bf m})\geq 0$ if $m_i\neq -1$ for all $i$. Here, the equality holds if and only if $m_i(m_i-1)=0$ for any $i$ and $m_im_j(m_im_j+1)=0$ for any $i\neq j$. Since we assume $\sum_{i=1}^n m_i=1$, this is equivalent to  $m_i=1$ for some $i$ and $m_k=0$ for other $k\neq i$.
  \end{proof}

By Lemma \ref{le3}, we restrict our attention to the case when $\widetilde{\bf m}$ has the form
\begin{align}\label{mr}
\widetilde{\bf m}=( \underbrace{-1,\ldots,-1}_{\alpha}, \underbrace{m_{\alpha+1},\ldots, m_{n}}_{n-\alpha})
\end{align}
for $m_{\alpha+1},\ldots, m_n\in \mathbb{Z}\setminus\{-1\}$ and $\alpha=1,\ldots n$. Here, we replaced the indices (if necessary)  so that $\widetilde{\bf m}$ has the form \eqref{mr}.  

\begin{lemma}\label{le4}
If $\widetilde{\bf m}$ has the form \eqref{mr} and $|m_i|>1$ for some $i\in \{\alpha+1,\ldots, n\}$, then $Q_{n,r}(\widetilde{\bf s},\widetilde{\bf m})>0$.
\end{lemma}
\begin{proof}
Since $m_i=-1$ for $i=1,\ldots, \alpha$ and $m_j\neq -1$ for $j=\alpha+1,\ldots, n$, the equation \eqref{q12} shows
\begin{align}
Q_{n,r}(\widetilde{\bf s},\widetilde{\bf m})&\geq \sum_{i=1}^{\alpha} \frac{-4\tanh^2r}{s_i}+2\sum_{i<j}\frac{m_im_j(m_im_j+1)}{s_is_j}\nonumber\\
&\geq \sum_{i=1}^{\alpha} \frac{-4\tanh^2r}{s_i}+2\sum_{i=1}^{\alpha}\sum_{j=\alpha+1}^n\frac{m_j(m_j-1)}{s_is_j}+\sum_{1<k<l<\alpha}\frac{4}{s_ks_l}\nonumber\\
&=2\Big(\sum_{i=1}^{\alpha} \frac{1}{s_i}\Big)\Big{\{}-2\tanh^2r+\sum_{j=\alpha+1}^{n}\frac{m_j(m_j-1)}{s_j}\Big{\}}+\sum_{1\leq k<l<\alpha}\frac{4}{s_ks_l}\label{q13}
\end{align}
Therefore, if there exists $m_j$ for $j=\alpha+1,\ldots, n$ satisfying  $|m_j|>1$, then  $m_j(m_j-1)/s_j>2>2\tanh^2r$, and hence, ${Q}_{n,r}(\widetilde{\bf s},\widetilde{\bf m})>0$.  
\end{proof}

Combining this lemma with $\sum_{i=1}^nm_i=1$, the remaining case is when 
\begin{align}\label{mr2}
\widetilde{\bf m}=( \underbrace{-1,\ldots,-1}_{\alpha},\underbrace{1,\ldots, 1}_{\alpha+1}, \underbrace{0,\ldots, 0}_{n-(2\alpha+1)})
\end{align}
for $\alpha=1,\ldots, [n/2]$.  Note that there is no $\widetilde{\bf m}$ of the form \eqref{mr2} for $n\leq 2$.

\begin{proposition}\label{p3}
Suppose $n\geq 3$ and $\widetilde{\bf m}$ is of the form \eqref{mr2}. Then, $Q_{n,r}(\widetilde{\bf s},\widetilde{\bf m})\geq 0$ if and only if
\begin{align}\label{ineq}
s_i\geq \tanh^2r\cdot s_js_k 
\end{align}
holds for any distinct $i,j,k\in\{1,\ldots, n\}$ with $\sum_{i=1}^ns_i=1$.
In particular, there exist infinitely many H-unstable torus  in $\mathbb{C}H^n$ when $n\geq 3$.
\end{proposition}
\begin{proof}
Suppose $\widetilde{\bf m}$ has the form \eqref{mr2}. If $\alpha\geq 2$, \eqref{q13} becomes
\begin{align*}
{Q}_{n,r}(\widetilde{\bf s},\widetilde{\bf m})&\geq -4\tanh^2r\Big(\sum_{k=1}^{\alpha} \frac{1}{s_k}\Big)+\sum_{1\leq k<l<\alpha}\frac{4}{s_ks_l}\\
&=-\frac{4\tanh^2r}{\alpha-1}\sum_{1\leq k<l<\alpha}\Big(\frac{1}{s_k}+\frac{1}{s_l}\Big)+\sum_{1\leq k<l<\alpha}\frac{4}{s_ks_l}
\\
&=\sum_{1\leq k<l<\alpha}\frac{4}{s_ks_l}\Big(1-\frac{\tanh^2r(s_k+s_l)}{\alpha-1}\Big)\\
&>0
\end{align*}
since $\tanh^2r(s_k+s_l)<1$. If $\alpha=1$, i.e., $\widetilde{\bf m}=(-1,1,1,0,\ldots,0)$, we have
\begin{align*}
{Q}_{n,r}(\widetilde{\bf s},\widetilde{\bf m})=\frac{4}{s_2s_3}-\frac{4\tanh^2r}{s_1},
\end{align*}
and this may be negative for  some ${\bf \tilde{s}}$. Since we replaced the  indices  so that $\widetilde{\bf m}=(-1,1,1,0,\ldots,0)$, this implies $Q_{n,r}(\widetilde{\bf s},\widetilde{\bf m})\geq 0$ if and only if the inequality
 \eqref{ineq} holds  for any distinct $i,j,k$ with $\sum_{i=1}^ns_i=1$. 
\end{proof}

Since we set $T^n_{r, \widetilde{\bf s}}=S^1(r_1)\times\cdots\times S^1(r_n)$ with $r_i:=\sinh r\sqrt{s_i}$ for $i=1,\ldots, n$ (see \eqref{param}), the inequality \eqref{ineq} is equivalent to
\begin{align}\label{ineq2}
\Big(1+\sum_{l=1}^n r_l^2\Big)^{1/2}r_i\geq r_jr_k,
\end{align}
where we used the relation $\sum_{i=1}^n s_i=1$. 
This proves the first assertion of Theorem \ref{maintheorem} (b). For example, if some $r_i$ is sufficiently small, then the inequality \eqref{ineq2} does not hold, and hence, the corresponding torus in $\mathbb{C}H^n$ is H-unstable.

Although there exist infinitely many H-unstable torus when $n\geq 3$, we can find an H-stable torus as follows:
The {\it Clifford torus} $T^n$ in $\mathbb{C}^n$ is the torus of the form $T^n=\{(re^{\sqrt{-1}\theta},\ldots, re^{\sqrt{-1}\theta}); e^{\sqrt{-1}\theta}\in S^1\}$ for $r\in (0,\infty)$. For this particular case, we prove

\begin{theorem}\label{cliff}
Let $T^n$ be the Clifford torus in $\mathbb{C}^n$ for $n\geq 1$. Then, $\Phi^{-1}(T^n)$ is Hamiltonian stable in $\mathbb{C}H^n$.
\end{theorem}

\begin{proof}
 In our notation described in the previous subsections, the Clifford torus is exactly the case when $s_1=\ldots=s_n=1/n$. We shall show $Q_{n,r}(\widetilde{\bf s}, {\bf m})\geq 0$.
 
First, we consider the case when $a_3({\bf m})=0$. Since $s_1=\ldots=s_n=1/n$, the inequality \eqref{ineq} is equivalent to $n\geq \tanh^2r$, and this holds for any $n\geq 1$. Combining this with Lemma \ref{le2} and \ref{le3}, we obtain $Q_{n,r}(\widetilde{\bf s}, {\bf m})\geq 0$ for $a_3({\bf m})=0$. 

Next, we consider the case when $a_3({\bf m})\neq 0$.  Setting $A:=\sum_{i=1}^nm_i^2$ and $B:=\sum_{i=1}^nm_i$, we see
\begin{align*}
Q_{n,r}(\widetilde{\bf s}, {\bf m})&=n^2(A^2+B^2-2A)-2\tanh^2r\cdot nB^2(A-1)+\tanh^4rB^2(B^2-1).
\end{align*}
If $n=1$, we have $B^2=A$, and hence, 
$
Q_{1,r}(\widetilde{\bf s}, {\bf m})
=(1-\tanh^2r)^2A(A-1)
\geq 0.
$
Here, the equality holds if and only if $A=m_1^2=0$ or $1$.
For $n\geq 2$, we estimate as follows:
\begin{align*}
Q_{n,r}(\widetilde{\bf s}, {\bf m})&=n^2(A^2+B^2-2A)-2\tanh^2r\cdot nB^2(A-1)+\tanh^4r\cdot B^2(B^2-1)\\
&=(nA-\tanh^2rB^2)^2-2n(nA-\tanh^2r B^2)+(n^2-\tanh^4r)B^2\\
&=\Big{\{}(nA-\tanh^2rB^2)-n\Big{\}}^2-n^2+(n^2-\tanh^4r)B^2\\
&\geq -n^2+(n^2-\tanh^4r)\cdot 4\\
&=3n^2-4\tanh^2r\\
&>0,
\end{align*}
where, in the first inequality, we used the fact $|B|\geq 2$ since $a_3({\bf m})\neq 0$.  This proves the theorem.
\end{proof}

\begin{remark}
{\rm We shall show in Subsection 4.4 below the pull-back of Clifford torus $\Phi^{-1}(T^n)$ is rigid, namely, $Q_{n,r}(\widetilde{\bf s}, {\bf m})=0$ if and only if the corresponding hamiltonian $u_{\bf m}^c$ or $u_{\bf m}^s$ generates an infinitesimal isometry on $\mathbb{C}H^n$. Therefore, one can find a torus orbit which is sufficiently close to $\Phi^{-1}(T^n)$ and H-stable in $\mathbb{C}H^n$ since $Q_{n,r}(\widetilde{\bf s}, {\bf m})$ is continuous with respect to $\widetilde{\bf s}$. In this sense, the H-stable torus orbit in $\mathbb{C}H^n$ is not unique. 
}
\end{remark}

\subsection{The case when $n=2$.} In this subsection, we consider another special situation, that is, when $n=2$. Note that Proposition \ref{p3} is not valid for this case. In fact, we prove the following result:
\begin{theorem}\label{n2}
Every Lagrangian torus orbits in $\mathbb{C}H^2$ is Hamiltonian stable.
\end{theorem}

\begin{proof}
We shall prove $Q_{2,r}(\widetilde{\bf s},{\bf m})\geq 0$. In the following, we simply write $a_i(\widetilde{\bf s},{\bf m})$ in $Q_{2,r}(\widetilde{\bf s},{\bf m})$ by $a_i$. Note that all coefficient $a_i$ are non-negative.   When $a_3=0$, the results in subsection 4.1 implies $Q_{2,r}(\widetilde{\bf s},{\bf m})\geq 0$ for any $r$, $\widetilde{\bf s}$ and ${\bf m}$. Thus,  we assume $a_3\neq0$, or equivalently,
\begin{align*}
m_1+m_2\neq 0\quad {\rm and}\quad m_1+m_2\neq \pm1
\end{align*}
in the lest of this proof.
 Moreover, if $m_i=0$ for some $i$,  the problem is reduced to the case when $n=1$, and this has already been considered in subsection 4.2. Therefore, we suppose $m_i\neq 0$ for $i=1,2$. Our claim is $Q_{2,r}(\widetilde{\bf s},{\bf m})>0$ (strictly positive) for such ${\bf m}$.

Since $a_3>0$ and
\begin{align*}
Q_{n,r}(\widetilde{\bf s},{\bf m})=a_3\Big(\tanh^2r-\frac{a_2}{a_3}\Big)^2-\frac{a_2^2}{a_3}+a_1,
\end{align*}
there are two possibilities: 
\begin{itemize}
\item[(i)] If $0<a_2/a_3<1$, then $Q_{n,r}(\widetilde{\bf s}, {\bf m})>0$ for any $r\in (0,\infty)$ if and only if 
\begin{align*}
a_1a_3-a_2^2>0.
\end{align*}
\item[(ii)] If $a_2/a_3\geq 1$, then $Q_{n,r}(\widetilde{\bf s}, {\bf m})> 0$ for any $r\in (0,\infty)$ if and only if 
\begin{align*}
a_1-2a_2+ a_3> 0.
\end{align*}
\end{itemize}

Let us consider the case (i). Then, we have $0<a_2<a_3$, and hence,
\begin{align*}
a_1a_3-a_2^2>a_1a_3-a_3^2=a_3(a_1-a_3).
\end{align*}
Thus, it is sufficient to prove  $a_1(\widetilde{\bf s}, {\bf m})-a_3({\bf m})> 0$. Since $s_1+s_2=1$, we consider a function for $s_1\in (0,1)$ by
\begin{align*}
&f(s_1):=a_1(\widetilde{\bf s}, {\bf m})
=\frac{\alpha_1}{s_1^2}+\frac{\alpha_2}{(1-s_1)^2}+2\beta\Big(\frac{1}{s_1}+\frac{1}{1-s_1}\Big),\\
&\textup{where $\alpha_i:=m_i^2(m_i^2-1)$ and $\beta:=m_1m_2(m_1m_2+1)$}.
\end{align*}
 Note that $\alpha_i\geq 0$ and $\beta\geq 0$, and $\alpha_1=\alpha_2=\beta=0$ if and only if $(m_1,m_2)=(1,-1)$ or $(-1,1)$ since $m_i\neq 0$. However,  this is not the case since $a_3\neq 0$. Thus, we may assume $\alpha_i>0$ or $\beta>0$ in the following. An elementary calculation shows that
\begin{align*}
\frac{\partial f}{\partial s_1}&=-\frac{2\alpha_1}{s_1^3}+\frac{2\alpha_2}{(1-s_1)^3}+2\beta\Big{\{}-\frac{1}{s_1^2}+\frac{1}{(1-s_1)^2}\Big{\}},\\
\frac{\partial^2 f}{\partial s_1^2}&=\frac{6\alpha_1}{s_1^4}+\frac{6\alpha_2}{(1-s_1)^4}+2\beta\Big{\{}\frac{2}{s_1^3}+\frac{2}{(1-s_1)^3}\Big{\}}.
\end{align*}
By assumptions, we have $\partial^2 f/\partial s_1^2>0$,  $\partial f/\partial s_1\to -\infty$ as $s_1\to 0$ and $\partial f/\partial s_1\to \infty$ as $s_1\to 1$, and hence, there exists a unique minimizer of the function $f(s_1)$ in the interval $(0,1)$. One can easily check that the minimizer is explicitly given by
\begin{align*}
s_1=\frac{m_1}{m_1+m_2},
\end{align*}
and 
\begin{align*}
\mathop{\rm min}_{0<s_1<1} f(s_1)=(m_1+m_2)^4.
\end{align*}
Therefore, we see
\begin{align*}
a_1-a_3\geq (m_1+m_2)^4-(m_1+m_2)^2\{(m_1+m_2)^2-1\}=(m_1+m_2)^2>0.
\end{align*}
Thus, we conclude $a_1a_2-a_3^2>0$ for the case (i).

Next, we consider the case (ii). Setting 
\begin{align*}
A:=\sum_{i=1}^2\frac{m_i^2}{s_i^2},\quad B:=\sum_{i=1}^2\frac{m_i^2}{s_i},\quad C:=\sum_{i=1}^2\frac{m_i}{s_i},\quad D:=\sum_{i=1}^2 m_i,
\end{align*}
we see
\begin{align*}
a_1-2a_2+a_3&=(B^2+C^2-2A)-2 (D^2B-DC)+D^2(D^2-1)\\
&=(B-D^2)^2-(C-D)^2+2(C^2-A)\nonumber\\
&=\Big{\{}\sum_{i=1}^2\Big(\frac{1}{s_i}-1\Big)m_i^2-2m_1m_2\Big{\}}^2-\Big{\{}\sum_{i=1}^2\Big(\frac{1}{s_i}-1\Big)m_i\Big{\}}^2+4\frac{m_1m_2}{s_1s_2}.\nonumber
\end{align*}
By using $s_1+s_2=1$, this is equivalent to
\begin{align}\label{t1}
s_1^2s_2^2(a_1-2a_2+a_3)=(s_2m_1-s_1m_2)^4-(s_2^2m_1+s_1^2m_2)^2+4s_1s_2m_1m_2.
\end{align}
We divide two cases: 

(ii-a) Suppose $m_1m_2<0$. The equation \eqref{t1} is rearranged as
\begin{align}\label{t2}
s_1^2s_2^2(a_1-2a_2+a_3)
&=s_2^4m_1^2(m_1^2-1)+s_1^4m_2^2(m_2^2-1)+6(s_1s_2)^2(m_1m_2)(m_1m_2+1)\\
&\quad -4s_1s_2m_1m_2(s_2^2m_1^2+s_1^2m_2^2+2s_1s_2-1).\nonumber
\end{align}
Notice that the former three terms in \eqref{t2} are non-negative since $m_1, m_2\in \mathbb{Z}$. On the other hand, since $m_1+m_2\neq 0$ and $m_1m_2<0$, we have $m_1^2+m_2^2>2$, and hence
\begin{align*}
s_2^2m_1^2+s_1^2m_2^2+2s_1s_2-1&=(m_1^2+m_2^2-2)s_1^2-2(m_1^2-1)s_1+(m_1^2-1)\\
&=(m_1^2+m_2^2-2)\Big(s_1-\frac{m_1^2-1}{m_1^2+m_2^2-2}\Big)^2+\frac{(m_1^2-1)(m_2^2-1)}{m_1^2+m_2^2-2}\\
&>0.
\end{align*} 
Combining this with $m_1m_2<0$, we see the last term of \eqref{t2} is strictly positive, and hence, we obtain $a_1-2a_2+a_3>0$. 

(ii-b) Suppose $m_1m_2>0$. We may assume $0<m_1\leq m_2$. We set
\begin{align*}
\gamma:=s_2m_1-s_1m_2\quad {\rm and}\quad \delta:=s_2m_1+s_1m_2
\end{align*}
so that $s_2m_1=(\delta+\gamma)/2$ and $s_1m_2=(\delta-\gamma)/2$. First, we assume $m_1>1$. Then, we estimate \eqref{t1} as follows:
\begin{align*}
s_1^2s_2^2(a_1-2a_2+a_3)
&=\gamma^4-\Big{\{}\frac{s_2}{2}(\delta-\gamma)+\frac{s_1}{2}(\delta+\gamma)\Big{\}}^2+(\delta^2-\gamma^2)\\
&=\gamma^4-\Big{\{}\frac{\delta}{2}+\frac{s_1-s_2}{2}\gamma\Big{\}}^2+(\delta^2-\gamma^2)\\
&\geq\gamma^4-2\Big{\{}\frac{\delta^2}{4}+\frac{(s_1-s_2)^2}{4}\gamma^2\Big{\}}+(\delta^2-\gamma^2)\\
&=\gamma^4-\Big(1+\frac{(s_1-s_2)^2}{2}\Big)\gamma^2+\frac{\delta^2}{2}\\
&=\Big{\{}\gamma^2-\frac{1}{2}\Big(1+\frac{(s_1-s_2)^2}{2}\Big)\Big{\}}^2-\frac{1}{4}\Big(1+\frac{(s_1-s_2)^2}{2}\Big)^2+\frac{\delta^2}{2}\\
&\geq \frac{\delta^2}{2}-\frac{9}{16}\\
&=\frac{1}{16}[8\{m_1+(m_2-m_1)s_1\}^2-9]\\
&\geq \frac{1}{16}(8m_1^2-9)\\
&>0.
\end{align*}
Here, in the second inequality, we used 
\begin{align*}
\frac{1}{4}\Big(1+\frac{(s_1-s_2)^2}{2}\Big)^2&=\frac{1}{4}\Big(1+\frac{(2s_1-1)^2}{2}\Big)^2\leq\frac{1}{4}\Big(1+\frac{1}{2}\Big)^2=\frac{9}{16}
\end{align*}
since $s_1+s_2=1$ and $0<s_1<1$. The third inequality is due to the assumption $m_2\geq m_1>0$. 
Finally, we consider the case when $1=m_1\leq m_2$. Then, by using \eqref{t2} and $s_1+s_2=1$, one easily verifies that
\begin{align*}
s_1^2s_2^2(a_1-2a_2+a_3)
&=s_1^2m_2(m_2+1)\Big\{(m_2+1)(m_2+2)\Big(s_1-\frac{2}{m_2+1}\Big)^2+2\frac{m_2-1}{m_2+1}\Big\}>0
\end{align*}
since $0<s_1<1$ and $m_2\geq 1$. Thus, $a_1-2a_2+a_3>0$ for the case (ii). This completes the proof of theorem.
\end{proof}

\subsection{Rigidity of H-stable torus}
Recall that an H-stable Lagrangian submanifold is called {\it rigid} if the null space of the second variation under the Hamiltonian deformations is spanned by holomorphic Killing vector fields.
In order to consider the rigidity of Lagrangian torus orbit, we need a lemma:  Let $\mathfrak{su}(1,n)$ be the Lie algebra of $SU(1,n)$ the group of holomorphic isometries on $\mathbb{C}H^n$, and  $\mathfrak{su}(1,n)=\mathfrak{k}\oplus \mathfrak{p}$ the Cartan decomposition, namely, we set
\begin{equation*}
\begin{aligned}
 \fk&=
 \Big{\{}
\left[
\begin{array}{c|ccc}
w & & &  \\ \hline
 & & & \\
 & & A & \\
& & & 
\end{array}
\right];
w\in \fu(1),\ A\in\fu(n), w+{\rm tr}_{\mathbb{C}}A=0
\Big{\}}
=\mathfrak{s}(\fu(1)\oplus \fu(n)),
 \\
 \fp&=
 \Big{\{}
\left[
\begin{array}{c|ccc}
 & & {}^{t}\overline{z} &  \\ \hline
 & & & \\
z & &  & \\
& & & 
\end{array}
\right];
z\in \mathbb{C}^n
\Big{\}}
\simeq \mathbb{C}^n,\\
\end{aligned}
 \end{equation*}
where $\mathfrak{k}$ is the Lie algebra of the maximal compact subgroup $K$. For an element $X\in \mathfrak{su}(1,n)$, the fundamental vector field $\widetilde{X}$ gives a holomorphic Killing vector filed on $\mathbb{C}H^n$. Conversely, any holomorphic Killing vector field on $\mathbb{C}H^n$ is obtained in this way. Since ${L}_{\widetilde{X}}\omega=0$,  Cartan's formula implies $i_{\widetilde{X}}\omega$ is a closed form, where $i$ denotes the inner product. Moreover, since $M=\mathbb{C}H^n$ is simply connected, there exists a Hamiltonian function $f\in C^{\infty}(M)$ so that $i_{\widetilde{X}}\omega=df$.  We shall explicitly determine the Hamiltonian function in our case.
For convenience, we count the number of row and column of matrix in $\mathfrak{su}(1,n)$ from $0$ to $n$, e.g. the $(0,0)$-component is the upper left component of the matrix.
We take a basis of $\fk$ by
\begin{align*}
X_{ij}^c:&=\sqrt{-1}(E_{i,j}+E_{j,i}),\quad X_{ij}^s:=E_{i,j}-E_{j,i}\quad  \textup{for}\ 1\leq i<j\leq n\quad {\rm and}\\
Z_{i}:&=\sqrt{-1}(E_{0,0}-E_{i,i}) \quad \textup{for}\ i=1,\ldots, n,
\end{align*}
and a basis of $\fp$ by
\begin{align*}
 X_i^c:=\sqrt{-1}(E_{i,0}-E_{0,i})\quad X_i^s:=E_{i,0}+E_{0,i}\quad \textup{for}\ i=1,\ldots, n,
\end{align*}
 where $E_{i,j}$ is the matrix unit.

\begin{lemma}
The Hamiltonian functions on $B^{n} $ for the fundamental vector fields $\widetilde{X}_{ij}^s, \widetilde{X}_{ij}^c, \widetilde{Z}_i, \widetilde{X}_i^s$ and $\widetilde{X}_i^c$ are given by
\begin{align*}
f_{ij}^c(z):&=\frac{{\rm Re}( z_i\overline{z}_j)}{1-|z|^2}, \quad f_{ij}^s(z):=\frac{{\rm Im}( z_i\overline{z}_j)}{1-|z|^2}, \quad  h_i:=\frac{1}{2}\cdot \frac{1+|z_i|^2}{1-|z|^2},\\
 f_i^c(z):&=\frac{{\rm Re} z_i}{1-|z|^2}\quad {\rm and}\quad f_i^s(z):=\frac{{\rm Im} z_i}{1-|z|^2},
\end{align*}
respectively.
\end{lemma}
One can check this lemma by a straightforward calculation. Thus, we omit the proof. 
\begin{proposition}
Let $\Phi^{-1}(T^n_{r,\tilde{s}})$ be an H-stable Lagrangian torus given in Theorem \ref{cliff} and \ref{n2}. Then, $\Phi^{-1}(T^n_{r,\tilde{s}})$ is rigid. 
\end{proposition}
\begin{proof}
By Lemma \ref{le1} through \ref{le4} and Proposition \ref{p3}, the null space of the second variation is spanned by the following functions:
\begin{align*}
u_{ij}^c:&=\cos(\theta_i-\theta_j),\quad u_{ij}^s:=\sin(\theta_i-\theta_j),\\
u_i^c:&={\cos\theta_i}, \quad u_i^s:={\sin\theta_i}, 
\end{align*}
for $i,j=1,\ldots, n$ with $i \neq j$. Recall that $\Phi^{-1}(T^n_{r,\tilde{s}})$ is contained in a geodesic hypersphere $S^{2n-1}_r=S^{2n-1}(\tanh r)$ in $B^{n}.$ For the fixed $r$,  we see
\begin{align*}
 u_{ij}^{\kappa}=(1-\tanh^2r) f_{ij}^{\kappa}|_{\Phi^{-1}(T^n_{r,\widetilde{s}})}\quad {\rm and}\quad u_i^{\kappa}=(1-\tanh^2r)f_{i}^{\kappa}|_{\Phi^{-1}(T^n_{r,\widetilde{s}})}
  \end{align*} for $\kappa=s$ or $c$. Therefore, the null vectors  $J\nabla_1 {u}_{i,j}^{\kappa}$ and $J\nabla_1 \tilde{u}_{i}^{\kappa}$ coincides with normal projections of some holomorphic Killing vector fields on $B^{n}$. Note that  $(\widetilde{Z}_i)^{\perp}=0$ along $\Phi^{-1}(T^n_{r,\tilde{s}})$. This proves the proposition.
\end{proof}

\subsection{Remarks on Hamiltonian volume minimizing property}
In this last section, we mention the Hamiltonian volume minimizing property for torus orbits.
\begin{definition}[cf. \cite{IO}]
{\rm
(1) A diffeomorphism $\Psi$ on a symplectic manifold $(M,\omega)$ is called {\it Hamiltonian} if $\Psi=\Psi_V^1$ for the flow $\Psi_V^t$ with $\Psi_V^0=Id_M$ of the time-dependent Hamiltonian vector field $V_t$ defined by a compactly supported Hamiltonian function $f_t\in C^{\infty}_c([0,1]\times M)$. We denote the set of Hamiltonian diffeomorphism by ${\rm Ham}_c(M,\omega)$. For Lagrangian submanifolds $L_0$ and $L_1$ in $M$, we say $L_1$ is {\it Hamiltonian isotopic} to $L_0$ if  there exists $\Psi\in{\rm Ham}_c(M,\omega)$ so that $L_1=\Psi(L_0)$.

(2) A Lagrangian submanifold $L$ in a almost K\"ahler manifold $(M, \omega, J, g)$ is called {\it Hamiltonian volume minimizing} if $L$ satisfies ${\rm Vol}_g(\Psi(L))\geq {\rm Vol}_g(L)$ for any $\Psi\in {\rm Ham}_c(M, \omega)$. 
}
\end{definition}

Let $\phi_1: L\rightarrow \mathbb{C}H^n$ be a $C(K)$-invariant Lagrangian embedding into $\mathbb{C}H^n(-4)$ and set $\phi_2:=\Phi\circ \phi_1: L\rightarrow \mathbb{C}^n$ as described in Section 3. Suppose $\phi_1(L)$ is contained in $S^{2n-1}_r$. Since the volume forms of $g_1:=\phi_1^*g$ and $g_2:=\phi_2^*g_{0}$ are related by $dv_{g_1}=\cosh r\cdot dv_{g_2}$, we have ${\rm Vol}_{g_1}(L)=\cosh r\cdot {\rm Vol}_{g_2}(L)$.

Consider the case when  $\phi_2(L)=T{(r_1,\ldots, r_n)}=S^1(r_1)\times\cdots\times S^1(r_n)\subset \mathbb{C}^n$. Since $\sum_{i=1}^n r_i^2=\sinh^2r$, we see
\begin{align}\label{tv}
{\rm Vol}_{g_1}(\Phi^{-1}(T{(r_1,\ldots, r_n)}))=(2\pi)^n\cdot \Big(1+\sum_{i=1}^n r_i^2\Big)^{1/2}\prod_{i=1}^nr_i.
\end{align}
Since $\Phi^{-1}$ preserves the Hamiltonian isotopy of $T{(r_1,\ldots, r_n)}$, the same argument described in Section 2 in \cite{IO} is valid for the case of torus orbits in $\mathbb{C}H^n$. Namely, setting $N(r_1,\ldots, r_n):=\sharp \{r_1,\ldots, r_n\}$, we see the following:

\begin{theorem}\label{nvol}
Suppose $n\geq 3$. If the inequality \eqref{ineq2} is not satisfied for some $i,j,k\in \{1,\ldots, n\}$ or $N(r_1,\ldots,r_{n})\geq 3$, then $\Phi^{-1}(T{(r_1,\ldots, r_n)})$ is not Hamiltonian volume minimizing in $\mathbb{C}H^n$.
\end{theorem} 
More precisely, if the inequality \eqref{ineq2} is not satisfied, the torus is H-unstable. If $N(r_1,\ldots,r_{n})\geq 3$, using the result of Chekanov \cite{Che}, we can find a torus 
\begin{align*}
T(r_1,\ldots, r_{j-1}, r_j',r_{j+1}\ldots,r_n)
\end{align*}
 so that $r_j'<r_j$ and is Hamiltonian isotopic to $T(r_1,\ldots, r_n)$ (see proof of Proposition 8 in \cite{IO}). Thus, by the formula \eqref{tv}, $T(r_1,\ldots, r_n)$ is not Hamiltonian volume minimizing in $\mathbb{C}H^n$.
In this sense, almost all Lagrangian torus orbits in $\mathbb{C}H^n$ are not Hamiltonian volume minimizing when $n\geq 3$, however, the following problem is still remaining  as well as the case of $\mathbb{C}^n$ and $\mathbb{C}P^n$:
\begin{problem}\label{prob}
Is $\Phi^{-1}(T(a,\ldots, a))$ Hamiltonian volume minimizing in $\mathbb{C}H^n?$
\end{problem}
When $n=1$, $\gamma_0:=\Phi^{-1}(T(a))$ is just a geodesic circle in the hyperbolic disk $B^2$  and 
a simple closed curve $\gamma$ on $B^2$ is Hamiltonian isotopic to $\gamma_0$ if and only if $A(\gamma)=A(\gamma_0)$, where $A(\gamma)$ is the area with respect to the hyperbolic metric of the region enclosed by $\gamma$. For a simple closed curve in $B^2$, we have the isoperimetric inequality on the hyperbolic disc;
\begin{align*}
\textup{length}(\gamma)^2\geq 4\pi A(\gamma)+A(\gamma)^2
\end{align*}
where the equality holds if and only if $\gamma=\gamma_0$ (cf. \cite{Osserman}). Thus,  the statement of Problem \ref{prob} is affirmative when $n=1$.

\subsection*{Acknowledgements} 
 The author would like to thank Yoshihiro Ohnita for a suggestion and his interest in this work. He also thanks to Hiroshi Iriyeh for useful comments and Takahiro Hashinaga for helpful discussion. This work was supported by JSPS KAKENHI Grant Number JP18K13420.

\end{document}